\documentclass[11pt]{amsart}

\usepackage[english]{babel}
\usepackage{amsmath, amsthm, amsfonts,stmaryrd,amssymb}
\usepackage{accents}
\usepackage{tikz-cd}
\tikzcdset{row sep/normal = 1cm}
\tikzcdset{column sep/normal = 1cm}
\usepackage{mathrsfs}
\usepackage{mathtools}
\usepackage{booktabs}
\usepackage{todonotes}
\usepackage[all]{xy}
\usepackage[utf8]{inputenc}
\usepackage{mathtools}
\usepackage{todonotes} 
\usepackage{framed}
\usepackage{graphicx}   
\usepackage[margin=3cm]{geometry}
\usepackage{hyperref}
\usepackage{xcolor}
\hypersetup{
    colorlinks=true,
    linkcolor=blue,  
    urlcolor=cyan,
    pdftitle={},
    pdfauthor={},
} 
\usepackage{theoremref}
\usepackage{subfig}
\usepackage[ruled,vlined]{algorithm2e}

\numberwithin{equation}{section}
\theoremstyle{plain}
\newtheorem{theorem}{Theorem}[section]
\newtheorem{remark}[theorem]{Remark}
\newtheorem{lemma}[theorem]{Lemma}

\newtheorem{proposition}[theorem]{Proposition}

\theoremstyle{definition}

\def\ed{\mathrm{d}}
 
\def\x{\times}

\let\bs=\boldsymbol
\let\mb=\mathbb
\let\mc=\mathcal

\def\x{\mathbf{x}}
\def\brho{\boldsymbol{\rho}}
\def\bF{\boldsymbol{F}}
\def\bFc{F}
\def\bphi{\boldsymbol{\phi}}
\def\blambda{\boldsymbol{\lambda}}

\def\Tc{\mathcal{T}}

\def\Bc{\mathcal{B}}
\def\Cc{\mathcal{C}}
\def\Ic{\mathcal{I}}
\def\Lc{\mathcal{L}}
\def\Kc{\mathcal{K}}
\def\Rc{\mathcal{R}}

\def\Jc{\mathcal{J}}

\newcommand\coarse[1]{{#1}'}

\begin{document}

\title{Computation of optimal transport with finite volumes}

\author[A. Natale]{Andrea Natale}
\address{Andrea Natale (\href{mailto:andrea.natale@u-psud.fr}{\tt andrea.natale@inria.fr}) Inria, Project team Rapsodi, Univ. Lille, CNRS, UMR 8524 - Laboratoire Paul Painlevé, F-59000 Lille, France
} 
\author[G. Todeschi]{Gabriele Todeschi}
\address{Gabriele Todeschi (\href{mailto:gabriele.todeschi@inria.fr}{\tt gabriele.todeschi@inria.fr}):  Inria Paris, Project team Mokaplan, Universit\'e Paris-Dauphine, PSL Research University, UMR CNRS 7534-Ceremade, 75016 Paris, France
} 

\begin{abstract}
We construct Two-Point Flux Approximation (TPFA) finite volume schemes to solve the quadratic optimal transport problem in its dynamic form, namely the problem originally introduced by Benamou and Brenier. We show numerically that these type of discretizations are prone to form instabilities in their more natural implementation, and we propose a variation based on nested meshes in order to overcome these issues. Despite the lack of strict convexity of the problem, we also derive quantitative estimates on the convergence of the method, at least for the discrete potential and the discrete cost. Finally, we introduce a strategy based on the barrier method to solve the discrete optimization problem.
\end{abstract}

\maketitle


\section{Introduction}\label{sec:intro}

The theory of optimal transport provides a robust way to define an interpolation between probability measures which takes into account the geometry of the space where they are defined. This theory is built around the problem of finding the optimal way of reallocating one given density into another, minimizing a total cost of displacement in space. The fundamental nature of such a  problem is responsible for the surprising links between optimal transport (and its generalizations) and physical models, most notably in fluid dynamics or via the theory of gradient flows, but also of its many applications  in social sciences or biology (see, e.g.,  \cite{santambrogio2015optimal} and references therein). Nowadays, several numerical methods are available to solve optimal transport problems and in particular to compute the associated interpolations between measures. However, only few of these can actually be generalized to more complex settings which are relevant for numerical modelling, and moreover their numerical analysis is often neglected. 
	
In this work we consider the numerical discretization of one of the most classical optimal transport problems in which the cost of displacement per unit mass is given by the square of the Euclidean distance. In particular, we consider finite volume discretizations of the so-called dynamical formulation of such a problem, following the approach originally proposed by Benamou and Brenier  \cite{benamou2000computational}. This formulation has inspired some of the first numerical methods for optimal transport, but it is still one of the most general, since it can be adapted easily to very complex settings. We will focus on three main aspects. Firstly, we will expose some numerical issues related to the stability of finite volumes methods that have been considered for this problem, and we propose a strategy based on nested meshes to overcome these. Secondly, we provide quantitative estimates on the convergence of the proposed methods to smooth solutions of the problem. Finally, we tackle the issue of the efficient computation of numerical solutions by applying and analyzing a classical interior point strategy adapted to our setting.


\subsection{Dynamical formulation}

	Consider a convex and compact domain $\Omega$.
	Given two densities $\rho^{in},\rho^f:\Omega \rightarrow [0,+\infty)$ with the same total mass, we consider the problem of finding a time-dependent density $\rho:[0,1] \times \Omega \rightarrow [0,+\infty)$ and a time-dependent momentum  $F: [0,1]\times \Omega \rightarrow \mathbb{R}^d$ solving
	\begin{equation}\label{eq:geod}
	\inf_{\rho,\bFc} \,\Bc(\rho,\bFc) 
	\end{equation}
	where $\rho$ and $\bFc$ satisfy the continuity equation
	\begin{equation}\label{eq:continuity}
	\begin{cases}
	\partial_t \rho+ \nabla \cdot \bFc = 0, \quad &\text{in } [0,1] \times \Omega,\\
	\bFc \cdot n_{\partial \Omega} = 0, \quad &\text{on }  [0,1] \times\partial \Omega,
	\end{cases}
	\end{equation}
	 with the further initial and final conditions $\rho(0,\cdot) = \rho^{in}, \rho(1,\cdot) = \rho^{f}$. The functional $\Bc(\rho,\bFc)$ is defined as follows:
	\begin{equation}\label{eq:BB}
	\Bc(\rho,\bFc)= \int_0^1 \int_{\Omega} B(\rho(t,\cdot),\bFc(t,\cdot)) \, \ed t \,,
	\end{equation}
	with $B:\mathbb{R}\times\mathbb{R}^d \rightarrow [0,+\infty]$ defined by
	\begin{equation}\label{eq:bb}
	B(p,Q) :=
	\begin{cases}
	\frac{|Q|^2}{2p}  &\text{if } p>0, \\
	0 &\text{if } p=0,\, Q=0, \\
	+\infty &\text{else}.
	\end{cases} 
	\end{equation}
	Problem (\ref{eq:geod}) selects the density interpolation between $\rho^{in}$ and $\rho^{f}$ which minimizes the total kinetic energy among all the non-negative solutions of the continuity equation \eqref{eq:continuity}. Note that the problem is written in the variables density-momentum rather than density-velocity, in order to obtain a convex formulation.
	
	Problem (\ref{eq:geod}) admits a dual formulation:
	\begin{equation}\label{eq:dual}
	\sup_{\phi} \int_{\Omega} \phi(1,\cdot)\, \rho^f - \int_{\Omega} \phi(0,\cdot)\, \rho^{in}\, ,
	\end{equation}
	where the potential $\phi: [0,1]\times \Omega \rightarrow \mathbb{R}$ satisfies the Hamilton-Jacobi equation
	\begin{equation}
	\partial_t \phi +\frac{1}{2} |\nabla \phi|^2 \le 0, \quad \text{in } [0,1] \times \Omega.
	\end{equation}
	Note that $\phi$ can be seen as the Lagrange multiplier of the continuity equation constraint \eqref{eq:continuity}. Problems (\ref{eq:geod})-(\ref{eq:dual}) coincide and their solution can be explicitly characterized as the solution to the system of primal-dual optimality conditions, namely:
	\begin{equation}\label{eq:geodesic}
	\begin{cases}
	\partial_t \rho+ \nabla \cdot (\rho \nabla \phi) = 0, \\
	\partial_t \phi +\frac{1}{2} |\nabla \phi|^2 \le 0,
	\end{cases}
	\end{equation}
	where $\bFc = \rho \nabla \phi$ is the optimal momentum and with the additional boundary conditions $\rho \nabla \phi \cdot n_{\partial \Omega} = 0$ on $\partial \Omega$, $\rho(0,\cdot) = \rho^{in}, \rho(1,\cdot) = \rho^{f}$. It is possible to show that the Hamilton-Jacobi equation can be saturated in problem \eqref{eq:dual} (using, e.g., the Hopf formula to characterize the solutions to the Hamitlon-Jacobi equation \cite{BARDI1984hopf}), i.e. the inequality can be replaced by the equality, and consequently also in system \eqref{eq:geodesic} by strong duality.

	Adapting appropriately the definitions above, problem \eqref{eq:geod} provides a notion of interpolation  between $\rho^{in}$ and $\rho^f$ when these latter are arbitrary probability measures. In this case the solution $\rho$ is itself a curve of probability measures which is generally denoted as Wasserstein interpolation (or geodesic), see, e.g.,  \cite{santambrogio2015optimal}.
	Moreover the minimum of the cost \eqref{eq:BB} coincides with half of the Wasserstein-2 distance squared between $\rho^{in}$ and $\rho^f$. More precisely, for a primal-dual solution $(\phi,\rho)$ to system (\ref{eq:geodesic}), this is given by:
	\[
	\frac{W_2^2(\rho^{in},\rho^f)}{2} = \int_0^1 \int_{\Omega} \frac{|\nabla \phi(t,\cdot)|^2}{2} \rho(t,\cdot) \,\ed t   = \int_{\Omega} \phi(1,\cdot) \rho^f   - \int_{\Omega} \phi(0,\cdot) \rho^{in} .
	\]
	
	
	\subsection{Discretization}
	
	In the original work of Benamou and Brenier \cite{benamou2000computational}, problem (\ref{eq:geod}) was discretized on regular grids using centered finite differences. Later in  \cite{papadakis2014optimal} Papadakis, Peyré and Oudet introduced a finite difference discretization using staggered grids, which are better suited for the discretization of the continuity equation. Similar finite differences approaches have been used also in more recent works \cite{carrillo2019primal,osher2018computations}. Note that the use of regular grids can be beneficial for the efficient solution of the scheme, but is not adapated to complex domains. Several finite elements approaches have been considered in order to construct schemes able to handle more general unstructured grids \cite{benamou2015augmented,benamou2016augmented,lavenant2018dynamical,nataleMFEforOT}. In particular in \cite{nataleMFEforOT} the authors proposed a $H(\mathrm{div})-$conforming finite element discretization that preserves at the discrete level the conservative form of the problem, in the same spirit of \cite{papadakis2014optimal}.
	
	Another approach to discretize problem (\ref{eq:geod}) is to use finite volumes, which is a natural choice given the conservative form of the constraint (\ref{eq:continuity}) and allows one to use unstructured grids.
	In \cite{erbarOTgraphs} Erbar, Rumpf, Schmitzer and Simon considered a discretization of problem \eqref{eq:geod} on graphs, which can be written under the formalism of Two-Point Flux Approximation (TPFA) finite volumes \cite{gladbach2018scaling}. They proved the Gamma-convergence of the discrete problem towards a semi-discrete version of \eqref{eq:geod}, discrete in space and continuous in time. In \cite{gladbach2018scaling}, Gladbach, Kopfer and Maas proved a convergence result for this semi-discretization towards the continuous problem. Combining these two results, it is possible to obtain a global convergence result, under conditions on the ratio between the temporal and spatial step sizes. Carrillo, Craig, Wang and Wei proved the Gamma-convergence without conditions on the step sizes but only for sufficiently regular and strictly positive solutions \cite{carrillo2019primal}. They used a centered finite difference discretization, which coincide with TPFA finite volumes on cartesian grids. Finally, in \cite{lavenant2019unconditional} Lavenant proved the weak convergence of discrete solutions (reconstructed as space-time measures) of a large class of time-space discretizations of \eqref{eq:geod}, unconditionally with respect to time and space steps and without assuming any regularity, and applied this result to the discretization studied in \cite{erbarOTgraphs}. The same result has been applied for example to the discretizations proposed in \cite{lavenant2018dynamical,nataleMFEforOT}.
	
	Our starting point in this work is the finite volume discretization presented in \cite{lavenant2019unconditional,erbarOTgraphs}.
	We observe numerically that for this discretization the density interpolation can exhibit oscillations which prevent strong convergence of the numerical solution, even when the exact interpolation is smooth.
	The same phenomenon has been observed by Facca and coauthors in \cite{facca1,facca2} when dealing with finite elements discretizations for the $L^1$ optimal transport problem, which is closely related to (\ref{eq:geod}).
	Our strategy to overcome this issue is inspired by these last works and consists in enriching the space of discrete potentials.
	We will show numerically that such a modification attenuates the oscillations and favors a stronger convergence. 
	
	Note that with this modification, the convergence result in  \cite{lavenant2019unconditional} cannot be applied straightforwardly. However, we will derive quantitative estimates for the convergence of the discrete Wasserstein distance and the discrete potential, which hold both in the enriched and original non-enriched case, in the case of smooth and strictly positive solutions.
	Even if such results are only partial as they do not apply to the density, they are still surprising given that the problem is not strictly convex. Moreover, we are not aware of similar estimates for the discretizations mentioned above. With these results at hand, it is possible to deduce again the weak convergence of the discrete density and momentum.

	
	\subsection{Numerical solution}\label{ssec:intro_num}
	
	A typical approach for solving discrete versions of the dynamical formulation (\ref{eq:geod}) is to apply first order primal dual methods. This goes back to the original paper of Benamou and Brenier \cite{benamou2000computational}, who proposed to use an Alternating Direction Method of Multipliers (ADMM) approach applied to the augmented Lagrangian of the discrete saddle point problem. Later \cite{papadakis2014optimal} considered different proximal splitting methods and recasted the previous algorithm into the same framework. Nowadays, these approaches are frequently used \cite{benamou2015augmented,benamou2016augmented,lavenant2018dynamical,carrillo2019primal,nataleMFEforOT}. In fact, they are robust and can take care automatically of the positivity of the density thanks to the definition of the function $\Bc$. Nevertheless, they are not easy to apply to arbitrary discretizations of the problem (especially on unstructured grids). More importantly, they are efficient only as far as high accuracy is not mandatory and uniform grids are used.

In the present work, we apply the so called barrier method, an instance of the wider class of interior point methods \cite{ConvOptBoyd,Gondzio25years,Terlaky,ForsGillWrig02}. The problem is perturbed by adding to the functional a strictly convex barrier function which repulses the density away from zero. In this way it reduces to an equality-constrained minimization problem, where the minimizer is automatically greater than zero and the objective functional is locally smooth around it, and which can be effectively solved using a Newton scheme. The perturbation introduced by the barrier function can be tuned by multiplying it by a positive coefficient $\mu$ and the original solution is recovered via a continuation method for $\mu$ going to zero.
The final algorithm is robust and can be easily generalized to similar problems (for example, we have already applied it successfully in \cite{nataleTPFAforGF} for the solution of Wasserstein gradient flows). 

A similar strategy has been applied by Achdou and coauthors \cite{achdouMFGnumerical} (although in the context of mean field games), perturbing the Lagrangian associated to the problem with the Dirichlet energies of the density and the potential. Such a perturbation does not ensure the positivity of the solution and this forces the use of a monotone discretization. Using a barrier function allows us to consider more general discretizations, with higher accuracy in space. 
The idea of using a regularization term to deal directly with the positivity constraint has been also explored in \cite{osher2018computations}, where the authors used the Fischer information as penalization term, but without considering a continuation method. In particular, the problem is solved for a fixed (small) value of the perturbation's parameter, leading to diffusive effects.
	
	
	\subsection{Structure of the paper}
	
	In section \ref{sec:scheme} we present the finite volume discretization of (\ref{eq:geod}): we set the notation for the partition of the domain $\Omega$, introduce the discrete operators and then define the discrete optimal transport problem. In section \ref{sec:convergence} we derive quantitative estimates on the convergence of the discrete distance and the discrete potential towards their continuous counterparts, under special hypotheses. In section \ref{sec:barrier} we present the barrier method, the strategy we employ to solve the discrete optimization problem as a sequence of simpler perturbed problems. We conclude with the presentation of few numerical results in order to assess the reliability of the scheme and verify the convergence results, in section \ref{sec:numerics}.
	
	
	\section{Finite Volume discretization}\label{sec:scheme}
	
	
	\subsection{The discretization of $\Omega$}\label{ssec:mesh}
	
	We assume the domain $\Omega \subset \mathbb{R}^d$ to be polygonal if $d=2$ or polyhedral if $d=3$, and we consider an admissible discretization for TPFA finite volumes ~\cite[Definition 9.1]{EGH00}.
	Cartesian grids, Delaunay triangulations or Vorono\"i tessellations are typical examples of admissible meshes in this sense. 
	We denote such a discretization as 
	$\left(\mathcal{T}, \overline{\Sigma}, {(\mathbf{x}_K)}_{K\in\mathcal{T}}\right)$, namely the ensemble of the set of polyhedral control volumes $K$, the set of faces $\sigma$ and the set of cell centers $\x_K$.
	The set $\overline{\Sigma}$ is composed of boundary faces $\Sigma_{ext} = \{ \sigma \subset \partial\Omega\}$ and internal faces $\sigma \in \Sigma = \overline{\Sigma} \setminus \Sigma_{ext}$. We denote by $\Sigma_{K} = \overline{\Sigma}_{K}\cap \Sigma$ the internal faces belonging to $\partial K$.
	The cell-centers $(\mathbf{x}_K)_{K\in\mathcal{T}} \subset \Omega$ are such that, if $K, L \in \mathcal{T}$ share a face $\sigma = K|L$, then the vector $\mathbf{x}_L-\mathbf{x}_K$ is orthogonal to $\sigma$ and has the same orientation as the normal $\bs{n}_{K,\sigma}$ to $\sigma$ outward with respect to  $K$.
	
	We denote the Lebesgue measure of $K\in\mathcal{T}$ by $m_K$.
	For each internal face $\sigma = K|L \in \Sigma$, we denote $m_{\sigma}$ its $(d-1)-$dimensional Lebesgue measure and we refer to the diamond cell as the polyhedron whose edges join $\mathbf{x}_K$ and $\mathbf{x}_L$ to the vertices of $\sigma$.
	Denoting by $d_\sigma = |\mathbf{x}_K-\mathbf{x}_L|$, the measure of the diamond cell is then equal to $m_\sigma d_\sigma/d$. We denote by $d_{K,\sigma}$ the Euclidean distance between the cell center $\mathbf{x}_K$ and the midpoint of the edge $\sigma \in \overline{\Sigma}_K$.
	In figure \ref{fig:FVcell} the notation is exemplified for a triangular element.
	
	We will need to distinguish between two different admissible discretizations of $\Omega$, where one is obtained as a subdivision of the other. We denote by $\left(\coarse{\Tc}, \overline{\coarse{\Sigma}}, {(\mathbf{x}_{\coarse{K}})}_{\coarse{K}\in \coarse{\Tc}}\right)$ the coarse mesh and by $\left(\Tc, \overline{\Sigma}, {(\mathbf{x}_{K})}_{K\in \Tc}\right)$ the fine one, and we require that
	\[
	\forall K \in \Tc,~ \exists \,\coarse{K} \in \coarse{\Tc} \text{ such that } \overline{K} \subseteq \coarse{\overline{K}}.
	\]
	In practice  we will consider two specific instances of this construction. The first is the trivial case where the two meshes coincide. The second holds at least in two dimensions and can be defined as follows. First, we take as coarse mesh a Delaunay triangulation, with cell centers $\x_{\coarse{K}}$ the circumcenters of each cell $\coarse{K}$. We further require that all the triangles are acute, so that all the cell centers $\x_{\coarse{K}}$ lie in the interior of the corresponding cell $\coarse{K}$.
	Then, we define the fine mesh by dividing each triangular cell $\coarse{K}$ into three quadrilaterals by joining $\x_{\coarse{K}}$ to the three midpoints of the edges $\coarse{\sigma}\in \coarse{\overline{\Sigma}}_{\coarse{K}}$. We take again as cell centers $\x_K$ of the fine mesh the circumcenters of each cell $K$. This construction is illustrated in figure \ref{fig:FVcell}. Note that the partition obtained in this way is indeed admissible. 
	
	\begin{figure}
		\centering
		\includegraphics[width=0.45\textwidth]{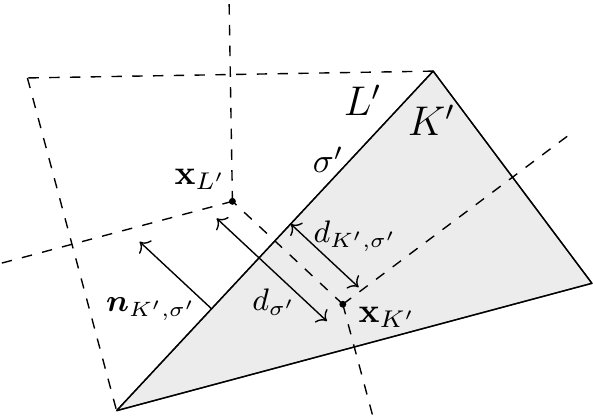} \quad
		\includegraphics[trim={-0.7cm -0cm -0.7cm -0cm},clip,width=0.45\textwidth]{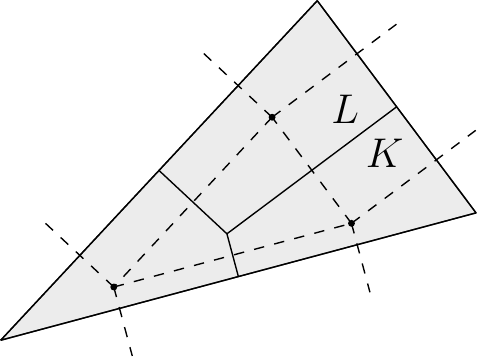}
		\caption{Exemplification of the notation of a triangular cell (left) and its subdivision (right).}
		\label{fig:FVcell}
	\end{figure}
	
	
	\subsection{Discrete spaces and operators}\label{ssec:hspaces}
	
	We introduce two discrete spaces defined on the two meshes, $\mathbb{P}_{\coarse{\mathcal{T}}} = \mathbb{R}^{\coarse{\mathcal{T}}}$ and $\mathbb{P}_{\mathcal{T}} = \mathbb{R}^{\mathcal{T}}$, each one endowed with its own weighted scalar product,
	\[
	\langle \cdot, \cdot \rangle_{\Tc}: (\boldsymbol{a},\boldsymbol{b})\in [\mathbb{P}_{\mathcal{T}}]^2 \mapsto \sum_{K\in\mathcal{T}} a_K b_K m_K\,,
	\]
	and similarly for $\langle \cdot, \cdot \rangle_{\coarse{\Tc}}$.
	Note that $\mathbb{P}_{\coarse{\Tc}}\subseteq \mathbb{P}_{\Tc}$, and we denote by $\Ic$ the canonical injection operator, which is given explicitly by
	\[
	\Ic : \mb{P}_{\coarse{\Tc}} \rightarrow \mb{P}_{\Tc}, \quad (\Ic \rho)_{K} = \rho_{\coarse{K}}, \quad \forall K \subset \coarse{K}.
	\]
	In the case where the two discretizations of $\Omega$ coincide, $\Ic$ is just the identity operator. We will denote by $\mathcal{I}^*$ the adjoint of $\mathcal{I}$, i.e.\ $	\langle  \mathcal{I}^* \cdot, \cdot \rangle_{\Tc'} = \langle \cdot, \mathcal{I} \cdot \rangle_{\Tc}$.
	We further introduce two discrete spaces defined on the finer mesh: the space $\mathbb{P}_{\Sigma} = \mathbb{R}^{\Sigma}$, defined on the diamond cells, endowed with the scalar product 
	\[
	\langle \cdot, \cdot \rangle_{\Sigma}: (\boldsymbol{u},\boldsymbol{v})\in [\mathbb{P}_{\Sigma}]^2 \mapsto \sum_{\sigma\in\Sigma} u_{\sigma} v_{\sigma} m_{\sigma} d_{\sigma} \,,
	\]
	and the space of discrete conservative fluxes,
	\begin{equation}\label{eq:space_fluxes}
	\mathbb{F}_{\mathcal{T}}=\{\boldsymbol{F}=(F_{K,\sigma},F_{L,\sigma})_{\sigma\in\Sigma}\in\mathbb{R}^{2\Sigma}: F_{K,\sigma}+F_{L,\sigma}=0\} \,,
	\end{equation}
	endowed with the scalar product 
	\[\langle \cdot, \cdot \rangle_{\mathbb{F}_{\Tc}}: (\bs{F},\bs{G})\in [\mathbb{F}_{\Tc}]^2 \mapsto \sum_{\sigma\in\Sigma} (F_{K,\sigma} G_{K,\sigma}+F_{L,\sigma} G_{L,\sigma}) \frac{m_{\sigma}d_{\sigma}}{2} \,.
	\]
  We denote by $\|\cdot\|_\Tc$, $\|\cdot\|_{\coarse{\Tc}}$, $\|\cdot\|_\Sigma$ and $\|\cdot\|_{\mathbb{F}_\Tc}$ the norms associated with the inner products defined above.   
  We denote $F_{\sigma} = |F_{K,\sigma}| = |F_{L,\sigma}|$ and we will use the convention  $|\bF| = (F_{\sigma})_{\sigma\in\Sigma} \in \mb{P}_{\Sigma}$ and $(\bF)^2 = (F_{\sigma}^2)_{\sigma\in\Sigma} \in \mb{P}_{\Sigma}$, for $\bF\in\mathbb{F}_{\mathcal{T}}$. Moreover, we define the element-wise multiplication by $\odot$. In particular, given $\bs{F},\bs{G} \in \mathbb{F}_\Tc$ and $\bs{u} \in \mathbb{P}_\Sigma$, we define $\bs{F} \odot \bs{G}, \bs{u} \odot \bs{F} \in \mathbb{F}_\Tc$ by
	\[
	[\bs{F} \odot \bs{G}]_{K,\sigma} \coloneqq F_{K,\sigma} G_{K,\sigma}\,, \quad  [\bs{u} \odot \bs{F}]_{K,\sigma} \coloneqq u_{\sigma} F_{K,\sigma}\,.
	\]
	
	We now introduce the discrete differential operators. The discrete divergence $\mathrm{div}_{\Tc}: \mathbb{F}_\Tc \rightarrow \mathbb{P}_\Tc$ is defined by
	$(\mathrm{div}_{\Tc} \bs{F})_K \coloneqq 	\mathrm{div}_{K} (\bs{F})$ where 
	\[
	\mathrm{div}_{K} \bs{F} \coloneqq  \frac{1}{m_K} \sum_{\sigma\in\Sigma_K} F_{K,\sigma} m_{\sigma}\,.
	\]
	The discrete gradient $\nabla_\Sigma: \mathbb{P}_{\Tc} \rightarrow \mathbb{F}_\Tc$ is defined by $\langle\nabla_\Sigma \bs{a}, \bs{F} \rangle_{\mathbb{F}_\Tc}= - \langle \bs{a} , \mathrm{div}_\Tc \bs{F}\rangle_{\mathbb{P}_\Tc}$. 
	In particular we also have $(\nabla_\Sigma \bs{a})_{K,\sigma} = 	\nabla_{K,\sigma} (\bs{a})$ where
	\[
	\mathrm{\nabla}_{K,\sigma} \bs{a} \coloneqq \frac{a_L -a_K}{d_\sigma} \,.
	\]
	Moreover, as for the discrete conservative fluxes, we define $\mathrm{\nabla}_{\sigma} \bs{a} \coloneqq |\mathrm{\nabla}_{K,\sigma} \bs{a}|$.

	We also need to introduce a reconstruction operator from cells to diamond cells $\Rc_{\Sigma}:\mathbb{P}_{\mathcal{T}}\rightarrow \mathbb{P}_{\Sigma}$, which will be required to construct the discrete energy.
	We require that the operator $\Rc_{\Sigma}$ be a concave function (component-wise), positively 1-homogeneous and positivity preserving. In practice, we will consider two weighted means, $\Lc_\Sigma$ and $\mathcal{H}_\Sigma$, which correspond respectively to a linear and a harmonic mean, and are defined as follows:
	\begin{equation}\label{eq:means}
	(\Lc_\Sigma \bs{a})_{\sigma}= \frac{d_{K,\sigma}}{d_{\sigma}} a_K +\frac{d_{L,\sigma}}{d_{\sigma}} a_L \,, \quad (\mathcal{H}_\Sigma \bs{a})_{\sigma}= \frac{d_\sigma a_K a_L}{d_{K,\sigma} a_L + d_{L,\sigma} a_K}\,,
	\end{equation}
	for any $\bs{a}\in\mathbb{P}_\Tc$. We  denote by $\text{d} \Rc_{\Sigma}[\bs{a}]:\mathbb{P}_{\mathcal{T}}\rightarrow \mathbb{P}_{\Sigma}$ the differential of $\Rc_{\Sigma}$ with respect to $\bs{a}$, evaluated at a given $\bs{a} \in \mathbb{P}_{\mathcal{T}}$. Clearly, if $\Rc_\Sigma = \Lc_\Sigma$, we simply have $\text{d} \Rc_{\Sigma}[\bs{a}] = \Lc_\Sigma$. Moreover, we denote by $(\text{d} \Rc_{\Sigma}[\bs{a}])^*$ the adjoint of $\text{d} \Rc_{\Sigma}[\bs{a}]$, with respect to the two different scalar products. For the two reconstructions we consider, this operator is given by either $\Lc_\Sigma^*$ or $(\text{d} \mathcal{H}_\Sigma[\bs{a}])^*$, which are defined by
	\begin{equation}\label{eq:adjoint_means}
	( \Lc_\Sigma^* \bs{u})_K = \sum_{\sigma\in\Sigma_K} u_{\sigma} \frac{m_{\sigma}d_{K,\sigma}}{m_K}\,, \quad ((\text{d} \mathcal{H}_\Sigma[\bs{a}])^* \bs{u})_{K}= \sum_{\sigma\in\Sigma_K} \frac{( \mathcal{H}_\Sigma[\bs{a}])^2_{\sigma}}{a^2_K} u_{\sigma} \frac{m_{\sigma}d_{K,\sigma}}{m_K}\,,
	\end{equation}
	for any $\bs{u}\in \mathbb{P}_{\Sigma}$.
	Finally, for any fixed $\bs{a} \in\mathbb{P}_{\coarse{\Tc}}$,  we define the reconstruction operator on the coarse grid $\mathcal{R}_{\coarse{\Tc}}[\bs{a} ]:  \mathbb{P}_{\Sigma}\rightarrow \mathbb{P}_{\Tc'}$ by 
	\begin{equation}\label{eq:rcoarse}
	\mathcal{R}_{\coarse{\Tc}} [\bs{a}] \coloneqq \Ic^* \circ (\ed \Rc_{\Sigma} [ \mathcal{I}(\bs{a})])^* \,.
	\end{equation}
	
	\begin{remark}\label{rmk:boundary_edges}
	The space of discrete conservative fluxes and the reconstruction operator introduced above take only into account the interior edges. This is sufficient for our purposes due to the zero flux boundary conditions. In particular, since the flux should be zero at the boundary the reconstruction of the density on the exterior edges is not needed for the construction of the scheme.
	\end{remark}

	
	\subsection{Discrete problem}\label{ssec:scheme}

	Consider a discretization of the time interval $[0,1]$ in $N+1$ subintervals of constant length $\tau=\frac{1}{N+1}$, and let $t^k \coloneqq k \tau $ for all $k \in \{0,..,N+1\}$. We denote the time evolution of a discrete density by $\bs{\rho}\coloneqq (\bs{\rho}^k)_{k=0}^{N+1}$, where $\bs{\rho}^k \coloneqq (\rho_{\coarse{K}}^k)_{{\coarse{K}}\in\coarse{\mc{T}}}$. Similarly we denote by $\bs{F}\coloneqq (\bs{F}^k)_{k=1}^{N+1}$ the time evolution of a discrete momentum, where $\bs{F}^k \coloneqq ( F_{K,\sigma}^k,F_{L,\sigma}^k )_{\sigma\in\Sigma}$.
	
	Given a couple $(\brho,\bF)\in[\mathbb{P}_{\coarse{\Tc}}]^{N+2}\times[\mathbb{F}_{\Tc}]^{N+1}$, we define the discrete equivalent of the functional (\ref{eq:BB}),  $\Bc_{N,\Tc}:[\mathbb{P}_{\Tc}]^{N+2}\times[\mathbb{F}_{\Tc}]^{N+1} \rightarrow \mathbb{R}^+$, as follows:
	\begin{equation}\label{eq:BBh}
		\Bc_{N,\Tc}(\brho,\bF) := 
		\begin{cases}
		\sum_{k=1}^{N+1} \tau \sum_{\sigma\in\Sigma} B(((\Rc_{\Sigma} \circ \Ic) (\frac{\brho^k+\brho^{k-1}}{2}))_\sigma,F_{\sigma}^k) m_{\sigma}d_{\sigma}  &\text{if } \rho^k_{\coarse{K}}\ge0, \\
		+\infty &\text{else},
		\end{cases} 
	\end{equation}
	where $B$ is defined in equation \eqref{eq:bb}. Since $\mathcal{R}_\Sigma$ is assumed to be concave, the function (\ref{eq:BBh}) is convex and lower semi-continuous.

	Note that on each subinterval $[\tau (k-1), \tau k]$, the time integral of the kinetic energy is discretized using the midpoint rule. This implies that a given $F^k_\sigma$ needs to vanish only if the reconstruction of $(\brho^k+\brho^{k-1})/2$ on the same edge vanishes. Approximating the integral with a left/right-endpoint approximation would be more restrictive in this sense (see \cite{lavenant2019unconditional} for more details on this choice of time discretization).
	At each time step, the kinetic energy is discretized on the diamond cells of the finer grid.
	The measure of each diamond cell is taken $d$ times. This is done in order to compensate the unidirectional discretization of the vector field $\bFc$ and therefore obtain a consistent discretization (see, e.g., lemma \ref{lem:properties}). Indeed, each $F_{\sigma}$ is meant as an approximation of $|\bFc\cdot \bs{n}_{\sigma}|$ and encodes then the information of $\bFc$ only along the direction $\bs{n}_{\sigma}$. This choice is also linked to the definition of inflated gradient (see \cite{CC2003,EyGa2003} for more details on this construction).
		
	\begin{remark}
	Note that \eqref{eq:BBh} is not simply the discretization of \eqref{eq:BB} on the diamond cells, in which case the functional would take the value $+\infty$ whenever the time-space reconstruction of the density is negative on some diamond cell. The functional in \eqref{eq:BBh} takes the value $+\infty$ whenever the density is negative on some cell $\coarse{K}\in\coarse{\Tc}$, which is a stronger condition.
	\end{remark}
	
	Given two discrete densities $\brho^{in}, \brho^f \in \mb{P}^+_{\coarse{\Tc}}$, with the same total discrete mass $\sum_{\coarse{K}\in\coarse{\Tc}} \rho_{\coarse{K}}^{in} m_{\coarse{K}} = \sum_{\coarse{K}\in\coarse{\Tc}} \rho_{\coarse{K}}^f m_{\coarse{K}}$, we consider the following discrete version of problem \eqref{eq:geod}:
	\begin{equation}\label{eq:geodth}
\inf_{(\brho,\bF)\in \Cc_{N,\Tc}} \Bc_{N,\Tc} (\brho,\bF) 
	\end{equation}
	where $\Cc_{N,\Tc} \subset  [\mathbb{P}_{\coarse{\Tc}}]^{N+2}\times[\mathbb{F}_{\Tc}]^{N+1}$ is the convex subset whose elements $(\brho,\bF)$ satisfy both the discrete continuity equation
	\begin{equation}\label{eq:continuityth}
\Ic ( \frac{\brho^k-\brho^{k-1}}{\tau})+  \mathrm{div}_{\Tc} \bs{F}^k = 0 \,,  \quad \forall k\in\{1,..,N+1\} , 
	\end{equation}
	and the initial and final conditions
	\begin{equation}\label{eq:timebc}
	\brho^0=\brho^{in}, \quad \brho^{N+1}=\brho^f.
	\end{equation}
	The continuity equation is discretized in time using the midpoint rule ($\bF$ is indeed staggered in time with respect to $\brho$). Moreover, given the definition of the discrete space of conservative fluxes and the operator $\mathrm{div}_{\Tc}$, \eqref{eq:continuityth} is to be understood with zero flux boundary conditions in space. Hence equations (\ref{eq:continuityth})-(\ref{eq:timebc}) imply that the total discrete mass is preserved. In the following, we explicitly enforce the constraint (\ref{eq:timebc}), i.e. we identify $\brho^0$ and $\brho^{N+1}$ with $\brho^{in}$ and $\brho^f$, respectively.
	
	We derive now the first order optimality conditions for problem (\ref{eq:geodth}), which are necessary and sufficient conditions for a solution. We consider the minimization on $\brho$ to be taken only among non-negative densities. The Lagrangian associated with the constrained optimization problem \eqref{eq:geodth} is given by
	\begin{equation}
	\Lc_{N,\Tc}(\bphi,\brho,\bF) = \Bc_{N,\Tc}(\brho,\bF)+ \sum_{k=1}^{N+1} \tau \langle \bphi^k, \Ic ( \frac{\brho^k-\brho^{k-1}}{\tau})+  \mathrm{div}_{\Tc} \bs{F}^k \rangle_{\Tc} \,,
	\end{equation}
	where the potential $\bphi \in [\mb{P}_{\Tc}]^{N+1}$ is the Lagrange multiplier for the continuity equation constraint. The stationarity condition of $\Lc_{N,\Tc}$ with respect to $\bF$ gives
	\begin{equation}\label{eq:fluxes}
	\bs{F}^k = (\Rc_{\Sigma}\circ \Ic)(\frac{\brho^k+\brho^{k-1}}{2}) \odot \nabla_\Sigma  \bphi^k\,,  \quad \forall k \in \{1,..,N+1\},
	\end{equation}
	so that the Lagrangian reduces to
	\begin{equation}\label{eq:lag_rhophi}
	-\frac{\tau}{2} \sum_{k=1}^{N+1} \langle (\Rc_{\Sigma} \circ \Ic) (\frac{\brho^k+\brho^{k-1}}{2}), (\nabla_{\Sigma} \bphi^k)^2 \rangle_{\Sigma} + \sum_{k=1}^{N+1} \tau \langle \bphi^k, \Ic ( \frac{\brho^k-\brho^{k-1}}{\tau}) \rangle_\Tc \,.
	\end{equation}
	A stationary point of \eqref{eq:lag_rhophi} must then satisfy the conditions:
	\begin{equation}\label{eq:geodth_KKTineq}
	\left\{
	\begin{aligned}
	&\displaystyle \Ic (\frac{\brho^k-\brho^{k-1}}{\tau}) + \mathrm{div}_{\Tc} ((\Rc_{\Sigma}\circ \Ic)(\frac{\brho^k+\brho^{k-1}}{2}) \odot   \nabla_{\Sigma} \bphi^k) =0,  \\
	&\displaystyle \Ic^*  (\frac{\bphi^{k+1}-\bphi^{k}}{\tau}) + \frac{1}{4}  \Rc_{\coarse{\Tc}} [\frac{\brho^k+\brho^{k-1}}{2}] (\nabla_\Sigma\bphi^k)^2 + \frac{1}{4} \Rc_{\coarse{\Tc}}[\frac{\brho^{k+1}+\brho^k}{2}](\nabla_\Sigma\bphi^{k+1})^2  \le 0,	
	\end{aligned}
	\right.
	\end{equation}
	where $k \in \{1,..,N+1\}$ for the discrete continuity equation, $k \in \{1,..,N\}$ for the discrete Hamilton-Jacobi equation, and where by equation \eqref{eq:rcoarse}, the linear operator $\Rc_{\coarse{\Tc}} [\frac{\brho^k+\brho^{k-1}}{2}]: \mathbb{P}_\Sigma \rightarrow \mathbb{P}_{\coarse{\Tc}}$ is defined by
	\[
		\Rc_{\coarse{\Tc}} [\frac{\brho^k+\brho^{k-1}}{2}] =  \Ic^* \circ ( \mathrm{d} \Rc_{\Sigma}[\Ic (\frac{\brho^k+\brho^{k-1}}{2})])^*.
	\]
	If $\Rc_\Sigma = \Lc_\Sigma$, then this operator does not depend on $\brho$ and in particular we will drop such dependency in the notation by setting $\Rc_{\coarse{\Tc}} =\Ic^* \circ \Lc^*_\Sigma$. We emphasize that the discrete no flux boundary conditions are automatically enforced by the definition of the discrete fluxes (see also remark \ref{rmk:boundary_edges}).
	
	The inequality in the second condition derives from the fact that the minimization in $\brho$ is taken over non-negative values, and the equality holds where $\brho^k$ does not vanish. Hence, we can write the full system of optimality conditions using a slack variable $\blambda\in[\mb{P}^+_{\coarse{\Tc}}]^N$:
	\begin{equation}\label{eq:geodth_KKT}
	\left\{
	\begin{aligned}
	&\displaystyle \Ic (\frac{\brho^k-\brho^{k-1}}{\tau}) + \mathrm{div}_{\Tc} ((\Rc_{\Sigma}\circ \Ic)(\frac{\brho^k+\brho^{k-1}}{2}) \odot   \nabla_{\Sigma} \bphi^k) =0 , \\
	&\displaystyle \Ic^* (\frac{\bphi^{k+1}-\bphi^{k}}{\tau})+ \frac{1}{4} \Rc_{\coarse{\Tc}} [\frac{\brho^k+\brho^{k-1}}{2}] (\nabla_\Sigma\bphi^k)^2 + \frac{1}{4} \Rc_{\coarse{\Tc}}[\frac{\brho^{k+1}+\brho^k}{2}](\nabla_\Sigma\bphi^{k+1})^2  = \blambda^k , \\
	&\rho^k_{\coarse{K}} \ge0, \, \lambda^k_{\coarse{K}} \le 0, \, \rho^k_{\coarse{K}} \lambda^k_{\coarse{K}} = 0,
	\end{aligned}
	\right.
	\end{equation}
	where $k \in \{1,..,N+1\}$ for the discrete continuity equation and  $k \in \{1,..,N\}$ for the other conditions. Note that system \eqref{eq:geodth_KKT} is a discrete version of the system of optimality conditions \eqref{eq:geodesic} holding at the continuous level. In particular, the continuity equation is discretized on the fine grid whereas the Hamilton-Jacobi equation on the coarse one. Using a discretization that preserves the monotonocity of the discrete Hamilton-Jacobi operator it is possible to show that the value zero for $\blambda$ is optimal (see \cite{LJKO} for a problem closely related to \ref{eq:geodth}), i.e. the discrete Hamilton-Jacobi equation can be saturated.  However this is not the case for the discretizations we consider since they do not preserve the monotonicity.

	\begin{remark}\label{rmk:Iidentity}
	If the two discretizations of $\Omega$ coincide, $\Ic$ becomes the identity and we recover the finite volumes discretization already considered in \cite{lavenant2019unconditional}, which is a fully discrete version of the continuous-time discrete optimal transport problem studied in \cite{gladbach2018scaling}.
	\end{remark}
	
	\begin{remark}\label{rmk:uniqueness}
	Existence of a (finite) solution $(\brho,\bF)$ is not difficult to obtain, as the minimization in $\brho$ is taken over a compact set and one can show that $|\bF|$ is uniformly bounded for any minimizing sequence (by the same arguments as in the proof of theorem \eqref{thm:muconvergence} below). 
	The uniqueness of the solution, which is guaranteed for the continuous problem (\ref{eq:geod}) as soon as the initial (or final) density is absolutely continuous with respect to the Lebesgue measure, is not evident. System (\ref{eq:geodth_KKT}) is not guaranteed in general to have a unique solution. In particular, where the density vanishes, the potential and the positivity multiplier are clearly non unique. The potential is however uniquely defined, up to a global constant, if the density solution is unique and everywhere strictly positive.
	\end{remark}
	
	Given a solution $(\brho, \bphi)$ to system (\ref{eq:geodth_KKT}), we can construct the associated momentum $\bF$ by equation \eqref{eq:fluxes} so that $(\brho,\bF)$ is a minimizer of problem \eqref{eq:geodth}. Then, we define the discrete Wasserstein distance $W_{N,\Tc}(\brho^{in}, \brho^f)$ by 
	\begin{equation}\label{eq:geodthW2}
	\frac{W_{N,\Tc}^2(\brho^{in}, \brho^f)}{2} \coloneqq \Bc_{N,\Tc} (\brho,\bF) .
	\end{equation}
	More precisely, replacing \eqref{eq:fluxes} in \eqref{eq:geodthW2}, the discrete Wasserstein distance can be computed using the following expression:
	\begin{equation}\label{eq:discW2pot}
	\frac{W_{N,\Tc}^2(\brho^{in}, \brho^f)}{2} = \frac{\tau}{2} \sum_{k=1}^{N+1} \langle (\Rc_{\Sigma} \circ \Ic) (\frac{\brho^k+\brho^{k-1}}{2}), (\nabla_{\Sigma} \bphi^k)^2 \rangle_{\Sigma}\,.
	\end{equation}
	
	In the case of the linear reconstruction, i.e.\ taking $\mc{R}_\Sigma = \mc{L}_\Sigma$, one can also easily express the dual to problem \eqref{eq:geodth} in terms of the potential $\bphi$, as in the continuous case, i.e.\ problem \eqref{eq:dual}. In fact, in this case, replacing the second condition of system \eqref{eq:geodth_KKT} into the Lagrangian (\ref{eq:lag_rhophi}) we obtain the following problem:
	\begin{equation}\label{eq:dualproblem}
\sup_{\bphi \in \tilde{\Kc}_{N,\Tc}} \langle \Ic^* \bphi^{N+1} - \frac{\tau}{4}  \mathcal{R}_{\coarse{\Tc}} (\nabla_\Sigma \bphi^{N+1})^2, \brho^f \rangle_\Tc - \langle \Ic^* \bphi^{1} + \frac{\tau}{4}  \mathcal{R}_{\coarse{\Tc}} (\nabla_\Sigma \bphi^{1})^2, \brho^{in} \rangle_\Tc
	\end{equation}
	where $\mathcal{R}_{\coarse{\Tc}} = \mc{I}^* \circ \mathcal{L}_{\Sigma}^*$ and $\Kc_{N,\Tc} \subset [\mathbb{P}_{\Tc}]^{N+1}$ is the convex subset of potentials $\bphi$ verifying
	\[
	\displaystyle \Ic^* (\frac{\bphi^{k+1}-\bphi^{k}}{\tau}) + \frac{1}{4} \mathcal{R}_{\coarse{\Tc}} ((\nabla_\Sigma\bphi^k)^2 +(\nabla_\Sigma\bphi^{k+1})^2 ) \leq 0\,.
	\]
	
	
\section{Convergence to the continuous problem}\label{sec:convergence}

In this section, we provide quantitative estimates for the convergence of the action and the discrete potential $\bs{\phi}$ towards their continuous counterparts, in the case of solutions with smooth strictly positive densities. Note that we restrict ourselves to the case of the linear reconstruction operator, i.e.\ we take $\Rc_{\Sigma} = \Lc_{\Sigma}$.
As a consequence of remark \ref{rmk:Iidentity}, these results are also valid for the finite volume discretization considered in \cite{lavenant2019unconditional}.
	
First of all, we introduce some additional notation.	
Let $\bs{F},\bs{G} \in [\mathbb{F}_\Tc]^{N+1}$ and $\bs{\rho}\in [\mathbb{P}_{\coarse{\Tc}}^+]^{N+2}$. We define the following weighted inner products:
\begin{equation}\label{eq:spacetimeprod}
\langle \bs{F},\bs{G}\rangle_{\bs{\rho}} \coloneqq  \tau \sum_{k=1}^{N+1} \langle \bs{F}^k,\bs{G}^k\rangle_{\frac{\bs{\rho}^k +\bs{\rho}^{k-1}}{2}} \, ,
\end{equation}
where
\[
\langle \bs{F}^k,\bs{G}^k\rangle_{\bs{\rho}^k} \coloneqq \sum_{\sigma\in \Sigma} (F^k_{K,\sigma} G^k_{K,\sigma} + F^k_{L,\sigma} G^k_{L,\sigma})((\mc{R}_\Sigma \circ \Ic) \bs{\rho}^k)_\sigma \frac{m_\sigma d_\sigma}{2}\,.
\]
We will denote by $\| \cdot \|_{\bs{\rho}}$ and $\| \cdot\|_{\bs{\rho}^k}$ the (semi-)norms associated with these (semi-)inner products.

We will consider two sampling operators: one for the density  $\Pi_{\coarse{\mc{T}}}:L^1(\Omega) \rightarrow \mathbb{P}_{\coarse{\mc{T}}} $, which performs an average on each cell, and one for the potential  $\Pi_{\mc{T}}:C^0(\Omega) \rightarrow \mathbb{P}_{\mc{T}}$, which evaluates the function at the cell center. More precisely, given $f\in L^1(\Omega)$ and $g\in C^0(\Omega)$, we define
\[
(\Pi_{\coarse{\mc{T}}} f)_{\coarse{K}} \coloneqq \frac{1}{m_{\coarse{K}}} \int_{\coarse{K}} f \,\ed x \,,\quad (\Pi_{\mc{T}} g)_K \coloneqq g(\x_K) \,,
\]
for all $\coarse{K} \in\coarse{\Tc}$ and $K\in \Tc$. 
For any time dependent functions $\rho \in C^0([0,1], L^1(\Omega))$ and $\phi \in C^0([0,1]\times\Omega)$  we define $\overline{\Pi}_{\coarse{\mc{T}}} {\rho}\coloneqq (\Pi_{\coarse{\mc{T}}} \rho(t^k,\cdot))_{k=0}^{N+1}$ and
\[
\overline{\Pi}_{\mc{T}} \phi \coloneqq \left(\frac{1}{\tau} \int_{t^{k-1}}^{t^{k}} \Pi_\mc{T} \phi(s,\cdot)\ed s \right)_{k=1}^{N+1}\,.
\]

We will denote by $h$ the maximum cell diameter of the fine mesh, i.e.\ $h \coloneqq \max_{K\in\Tc} \mathrm{diam}(K)$.
We will assume two regularity conditions on the fine mesh.
Firstly, there exists a constant $\zeta$, which does not depend on $h$, such that
\begin{equation}\label{eq:mesh_reg}
\mathrm{diam}(K) \leq \zeta d_{\sigma} \leq \zeta^2 \mathrm{diam}(K), \quad \forall \sigma \in \Sigma_K, \; \forall K\in\Tc \,;
\end{equation}
\begin{equation}\label{eq:mesh_reg1}
\mathrm{dist}(\x_K,K) \leq \zeta \,\mathrm{diam}(K), \quad \forall K\in\Tc \,.
\end{equation}
Secondly, there exists a constant $\eta_h>0$ only depending on $h$, with $\eta_h\rightarrow 0$ for $h\rightarrow 0$, such that
\begin{equation}\label{eq:asyiso}
\sum_{\sigma \in \Sigma_K} m_\sigma d_{K,\sigma}  \bs{n}_{K,\sigma} \otimes \bs{n}_{K,\sigma}   \leq  m_K ( 1+ \eta_h ) \mathrm{Id} \,, \quad \forall K\in\Tc \,.
\end{equation}
The latter condition is essentially a specific instance of the asymptotic isotropy condition in \cite{gladbach2018scaling} (see Definition 1.3). 
When the cell centers $\x_K$ are chosen as the circumcenters of the associated cell (as in the particular examples of meshes described in section \ref{ssec:mesh}), a stronger property holds, which has been referred to as center of mass condition \cite{gladbach2018scaling} or superadmissibility \cite{Eymard2009admissible}, and which reads as follows:
\begin{equation}\label{eq:asyiso2}
\sum_{\sigma \in \Sigma_K} m_\sigma d_{K,\sigma}  \bs{n}_{K,\sigma} \otimes \bs{n}_{K,\sigma}   =  m_K  \mathrm{Id} \,.
\end{equation}
However, for generality of the discussion, in the following we will only require \eqref{eq:asyiso} and therefore we will keep the dependence on $\eta_h$ explicit.

The following lemma collects some consistency properties of the projection ${\Pi}_{\mc{T}}$. In particular, point (3) below shows that the asymptotic isotropy condition implies the consistency of the quadratic term in the discrete Wasserstein distance \eqref{eq:discW2pot}, and justifies our discretization of the functional $\Bc_{N,\Tc}$. 

\begin{lemma}\label{lem:properties} The following properties hold:
\begin{enumerate}
\item for any $\psi\in C^0(\Omega)$,  $\max_{K\in \Tc} |(\Pi_\mathcal{T} \psi)_{K} | \leq \|\psi\|_{C^0}$; 
\item for any $\psi \in C^{0,1}(\Omega)$, there exists a constant $C>0$ only depending on $\psi$ and $\zeta$ such that
\[
\max_{K\in\Tc} \| ({\Pi}_{\mc{T}} \psi)_K - \psi \|_{C^0(K)}\leq C h\,;
\] 
\item  for any $\psi \in C^{1,1}(\Omega)$, there exists a constant $C>0$ only depending on $\psi$ and $\zeta$ such that
\[
 (\mathcal{L}_\Sigma^* |\nabla_\Sigma  {\Pi}_{\mc{T}} \psi|^2 )_K \leq (\Pi_{\Tc} |\nabla \psi|^2)_K + C (h +\eta_h) \,, 
\] 
for all $K\in \Tc$, where $\mc{L}_\Sigma$ is the linear reconstruction operator and $\eta_h$ is defined as in \eqref{eq:asyiso}.
\end{enumerate}
\end{lemma}

\begin{proof}
The first two points follow easily from the definition of $\Pi_\Tc$ and the regularity condition \eqref{eq:mesh_reg1}. For (3), observe that, by definition of the linear reconstruction operator,
\begin{equation}\label{eq:lstargrad}
(\mathcal{L}^*_\Sigma |\nabla_\sigma  {\Pi}_{\mc{T}} \psi|^2 )_K =  \sum_{\sigma \in \Sigma_K} |\nabla_\sigma  {\Pi}_{\mc{T}} \psi|^2 \frac{m_\sigma d_{K,\sigma}}{m_K} \,.
\end{equation}
Then, using the definition of the operator $\Pi_\Tc$ and the regularity condition \eqref{eq:mesh_reg},
\[
\begin{aligned}
\nabla_\sigma  {\Pi}_{\mc{T}} \psi =  \left| \frac{\psi(\x_K)-\psi(\x_L)}{d_{\sigma}} \right|
=\frac{1}{d_\sigma} \left| \int_0^1 \frac{\ed}{\ed s} \psi ((1-s) \x_K + s \x_L) \ed s \right|
\leq \left|  \nabla \psi( x_K) \cdot \bs{n}_{K,\sigma} \right| + C h\,.
\end{aligned}
\]
Replacing this into \eqref{eq:lstargrad}, neglecting higher order terms, and using the asymptotic isotropy assumption \eqref{eq:asyiso}, we obtain the desired bound.
\end{proof}

Propostion \ref{prop:elliptic} below is an adaptation to our setting of standard approximation results for elliptic problems. It quantifies the consistency of the projection $\overline{\Pi}_{\coarse{\mc{T}}}$ in terms of the associated potential. As in \cite{gladbach2018scaling}, we will use it to construct an admissible competitor for the discrete problem. Before proving the result, we state the following classical finite-volume version of the Poincaré inequality.

\begin{lemma}[Discrete mean Poincar\'e inequality, Lemma 10.2 in \cite{EGH00}] \label{lem:poincare}
There exists a constant $C>0$, only depending on $\Omega$, such that for all admissible
meshes $\Tc$, and for all $\bs{\psi} \in \mathbb{P}_\Tc$, the following inequality holds:
\[
\| \bs{\psi} - \frac{1}{|\Omega|} \sum_{K\in\Tc} \psi_K m_K \|_{\Tc} \leq C \|\nabla_\Sigma \bs{\psi}\|_{\mathbb{F}_\Tc} \,.
\]
\end{lemma}

\begin{proposition}\label{prop:elliptic}
Suppose that $\rho,\partial_t \rho\in L^{\infty}([0,1],C^{0,1}(\Omega))$, with $\rho\geq \varepsilon >0$, and let $\phi \in L^{\infty}([0,1],C^{1,1}(\Omega))$ be a solution of
\begin{equation}\label{eq:laplacian}
 - \mathrm{div}(\rho \nabla \phi) = \partial_t \rho  \,,\quad \nabla \phi \cdot n_{\partial \Omega} =   0\, \text{ on } \partial \Omega\,.
\end{equation}
Let $\bs{\rho} = \overline{\Pi}_{\coarse{\mc{T}}} \rho$ and let $\bs{\phi}$ be a solution of 
\[
- \mathrm{div}_{\Tc} ((\Lc_{\Sigma}\circ \Ic)(\frac{\brho^k+\brho^{k-1}}{2}) \odot   \nabla_{\Sigma} \bphi^k) =\Ic (\frac{\brho^k-\brho^{k-1}}{\tau}) \,.
\]
Then, there exists a constant $C>0$ depending only on $\phi$, $\rho$, $\varepsilon$, $\zeta$ and $\Omega$, such that
\begin{equation}\label{eq:ellipticest}
\|\nabla_\Sigma \bs{\phi}\|_{\bs{\rho}}^2 \leq  \int_{0}^{1} \int_\Omega  {\rho} |\nabla \phi|^2\, \ed x  \ed t + C(h + \tau + \eta_h) \,,
\end{equation}
with $\eta_h$ defined as in \eqref{eq:asyiso}.

\end{proposition}

\begin{proof}
First, we integrate equation \eqref{eq:laplacian} over the time-space cell $[t^{k-1},t^k]\times K$ and divide it by $\tau m_K$. This yields 
\begin{equation}\label{eq:laplaciand}
 - \mathrm{div}_{K} \bs{u}^k  =  \frac{1}{m_K \tau}  \int_K \int_{t^{k-1}}^{t^k} \partial_t \rho\,  \ed t\ed x \,.
\end{equation}
where $\bs{u} \in [\mathbb{F}_\Tc]^{N+1}$ is defined by
\[
u_{K,\sigma}^k \coloneqq \frac{1}{\tau m_\sigma} \int_\sigma \int_{t^{k-1}}^{t^k} (\rho \nabla \phi)  \cdot \bs{n}_{K,\sigma}\,\ed t \,\ed s \, .
\]
We define $\bs{e} \in [\mathbb{F}_\Tc]^{N+1}$ and $\bs{r} \in [\mathbb{P}_\Tc]^{N+1}$ by
\[
e_{K,\sigma}^k = u^k_{K,\sigma} - ((\Lc_\Sigma \circ \Ic)\, \frac{\brho^k+\brho^{k-1}}{2})_\sigma \nabla_{\sigma} (\overline{\Pi}_\Tc \phi)^k\,, \]
and denoting by $\coarse{K}$ the cell in $\coarse{\Tc}$ such that $K\subset \coarse{K}$, 
\[
r_{K}^k \coloneqq  \frac{1}{m_K \tau}  \int_K \int_{t^{k-1}}^{t^k} \partial_t \rho  \, \ed t\ed x -  \frac{1}{m_{\coarse{K}} \tau}  \int_{\coarse{K}} \int_{t^{k-1}}^{t^k} \partial_t \rho \, \ed t\ed x\,.
\]
Then
\[
- \mathrm{div}_{\Tc} ((\Lc_{\Sigma}\circ \Ic)(\frac{\brho^k+\brho^{k-1}}{2}) \odot   \nabla_{\Sigma} (\bphi^k - (\overline{\Pi}_\Tc \phi)^k ) = \bs{r}^k - \mathrm{div}_{\Tc} \bs{e}^k \,.
\]
Multiplying both sides by $(\bphi^k - (\overline{\Pi}_\Tc \phi)^k )$ we obtain
\[
\| \nabla_{\Sigma} (\bphi^k - (\overline{\Pi}_\Tc \phi)^k ) \|^2_{\frac{\brho^k+\brho^{k-1}}{2}} = \langle  \bs{r}^k - \mathrm{div}_{\Tc} \bs{e}^k,(\bphi^k - (\overline{\Pi}_\Tc \phi)^k )  \rangle_{\Tc} \,.
\]
Using the discrete Poincar\'e inequality of lemma \ref{lem:poincare} and the lower bound on $\rho$, this implies
\[
\| \nabla_{\Sigma} (\bphi^k - (\overline{\Pi}_\Tc \phi)^k ) \|_{\frac{\brho^k+\brho^{k-1}}{2}} \leq C ( \|  \bs{r}^k\|_\Tc +\| \bs{e}^k\|_{\mathbb{F}_\Tc} )\,,
\]
where $C>0$ is a constant only depending on the lower bound $\varepsilon$ and the domain. By the regularity of $\phi$ and $\rho$, and the estimate \eqref{eq:mesh_reg}, we then obtain
\begin{equation}\label{eq:boundphih}
\| \nabla_{\Sigma} (\bphi^k - (\overline{\Pi}_\Tc \phi)^k ) \|_{\frac{\brho^k+\brho^{k-1}}{2}} \leq C ( h + \tau)\,,
\end{equation}
where now $C$ depends also on $\rho$ and $\phi$.

In order to get an estimate on the energy, we observe that $\bs{\phi}^k$ minimizes the functional
\[
\bs{\psi} \in [\mathbb{P}_{\Tc}]^{N+1} \; \longmapsto \; \| \nabla_{\Sigma} \bs{\psi} \|^2_{\frac{\brho^k+\brho^{k-1}}{2}} - \langle \Ic (\frac{\brho^k-\brho^{k-1}}{\tau}) , \bs{\psi} \rangle_\Tc \,,
\]
which implies the inequality
\[
\| \nabla_{\Sigma} \bphi^k  \|^2_{\frac{\brho^k+\brho^{k-1}}{2}} \leq \| \nabla_{\Sigma}  (\overline{\Pi}_\Tc \phi)^k ) \|^2_{\frac{\brho^k+\brho^{k-1}}{2}} + \langle \Ic (\frac{\brho^k-\brho^{k-1}}{\tau}) , (\bphi^k - (\overline{\Pi}_\Tc \phi)^k ) \rangle_{\Tc} \,.
\]
Using again the discrete Poincar\'e inequality of lemma \ref{lem:poincare} and the lower bound on $\rho$, as well as its regularity, we get
\[
\| \nabla_{\Sigma} \bphi^k  \|^2_{\frac{\brho^k+\brho^{k-1}}{2}} \leq \| \nabla_{\Sigma}  (\overline{\Pi}_\Tc \phi)^k ) \|^2_{\frac{\brho^k+\brho^{k-1}}{2}} + C \| \nabla_{\Sigma} (\bphi^k - (\overline{\Pi}_\Tc \phi)^k ) \|_{\frac{\brho^k+\brho^{k-1}}{2}} \,.
\]
Hence, using \eqref{eq:boundphih}, we obtain
\[
\| \nabla_{\Sigma} \bphi  \|^2_{\brho} \leq \| \nabla_{\Sigma} \overline{\Pi}_\Tc \phi ) \|^2_{\brho} + C (h+\tau)\,.
\]
Finally, using Jensen's inequality and then lemma  \ref{lem:properties}, we find
\[
\begin{aligned}
 \| \nabla_{\Sigma} \overline{\Pi}_\Tc \phi \|^2_{\brho} & \leq \sum_{k=1}^{N+1} \int_{t^{k-1}}^{t^k}  \| \nabla_{\Sigma} {\Pi}_\Tc \phi(t,\cdot) ) \|^2_{\frac{\brho^{k}+\brho^{k-1}}{2}}  \ed t \\
& \leq \sum_{k=1}^{N+1} \int_{t^{k-1}}^{t^k}  \langle \Ic \frac{\brho^{k}+\brho^{k-1}}{2}, \Pi_{\Tc} |\nabla \phi(t,\cdot)|^2 \rangle_\Tc \ed t + C(h + \eta_h)\\
& = \sum_{k=1}^{N+1} \sum_{K\in \Tc} \int_{t^{k-1}}^{t^k} \int_K \frac{\rho(t^{k},\cdot)+\rho(t^{k-1},\cdot)}{2} |\nabla \phi(t,\x_K)|^2\, \ed x  \ed t+ C(h + \eta_h)\\
& \leq  \int_{0}^{1} \int_\Omega  {\rho} |\nabla \phi|^2\, \ed x  \ed t + C(h + \tau + \eta_h)\, ,
\end{aligned}
\]
which concludes the proof.
\end{proof}

We are now ready to state the two main convergence results of this section, which provide quantitative estimates for the convergence rates of the discrete action and the discrete potential.

\begin{theorem}[Convergence of the action]\label{th:action}
Suppose that $\phi: [0,1]\times \Omega\rightarrow \mathbb{R}$ is an optimal potential for the dual Wasserstein problem from $\rho^{in}$ to $\rho^f$ and that $\rho:[0,1]\times \Omega \rightarrow [0,+\infty)$ is the associated interpolation. Then, denoting $\brho^{in} \coloneqq\Pi_{\coarse{\mc{T}}} \rho^{in}$ and $\brho^f\coloneqq \Pi_{\coarse{\mc{T}}} \rho^f$, and taking $\eta_h$ as in \eqref{eq:asyiso}, the following holds:
\begin{enumerate}
\item if $\phi \in C^{1,1}([0,1]\times \Omega)$, there exists a constant $C>0$ only dependent on $\phi$ and $\zeta$ such that 
\[
W_{N,\Tc}^2(\brho^{in}, \brho^f) \geq W_2^2(\rho^{in},\rho^f)   - C (h +\tau+ \eta_h) \,;
\]  
\item if $\phi \in L^{\infty}([0,1],C^{1,1}(\Omega))$ and $\rho,\partial_t \rho\in L^{\infty}([0,1],C^{0,1}(\Omega))$, with $\rho\geq \varepsilon >0$, there exists a constant $C>0$ depending only on $\rho$, $\phi$, $\varepsilon$, $\zeta$ and $\Omega$ such that
\[
W_{N,\Tc}^2(\brho^{in}, \brho^f) \leq W_2^2(\rho^{in},\rho^f)   + C (h +\tau+ \eta_h) \,.
\]  
\end{enumerate}
\end{theorem}

\begin{proof}
For the first point, we first observe that by lemma \ref{lem:properties} and the regularity of $\phi$, $\overline{\Pi}_\Tc \phi$ verifies
	\[
	\displaystyle \Ic^* (\frac{(\overline{\Pi}_\Tc \phi)^{k+1}-(\overline{\Pi}_\Tc \phi)^{k}}{\tau}) + \frac{\tau}{4} \Rc_{\coarse{\Tc}}((\nabla_\Sigma(\overline{\Pi}_\Tc \phi)^k)^2 +(\nabla_\Sigma(\overline{\Pi}_\Tc \phi)^{k+1})^2 ) \leq C (h + \tau + \eta_h)\,.
	\]
Then, define $\bphi$ by $\bphi^{k} \coloneqq (\overline{\Pi}_\Tc   \phi)^k - C (t^k+ t^{k-1}) (h + \tau + \eta_h)/2$, for $k\in\{1,\ldots,N+1\}$. Then $\bphi$ is admissible for the dual problem \eqref{eq:dualproblem}, hence
\[
	\frac{W_{N,\Tc}^2(\brho^{in}, \brho^f)}{2} \geq  \langle \displaystyle \Ic^* \bphi^{N+1} - \frac{\tau}{4} \mathcal{R}_{\coarse{\Tc}} (\nabla_\Sigma \bphi^{N+1})^2, \brho^f \rangle_{\coarse{\Tc}} - \langle \displaystyle \Ic^* \bphi^{1} + \frac{\tau}{4} \mathcal{R}_{\coarse{\Tc}} (\nabla_\Sigma \bphi^{1})^2, \brho^{in} \rangle_{\coarse{\Tc}} \, .
\]
Replacing back the definition of $\bphi$ and using the fact that $|\nabla_\sigma \bphi^1|$ and  $|\nabla_\sigma \bphi^{N+1}|$ are uniformly bounded by a constant  depending only on $\phi$, we get
\[
\frac{W_{N,\Tc}^2(\brho^{in}, \brho^f)}{2} \geq \int_\Omega \phi(1,\cdot) \rho^f - \int_\Omega \phi(0,\cdot) \rho^{in}   - C (h +\tau+ \eta_h) \, .
\]  
	
For the second point it suffices to observe that the couple $(\rho, \phi)$ satisfies \eqref{eq:laplacian}. Then, defining $\bs{\rho}$ and $\bs{\phi}$ as in the statement of proposition \ref{prop:elliptic}, we can construct an admissible competitor $(\bs{\rho},\bs{F})$ for the discrete optimal transport problem by defining the momentum $\bs{F}\in[\mathbb{F}_\Tc]^{N+1}$ as in equation \eqref{eq:fluxes}. Since, by definition,
\[	W_{N,\Tc}^2(\brho^{in}, \brho^f) \leq 2\,\Bc_{N,\Tc} (\bs{\rho},\bs{F}) = \| \nabla_\Sigma \bs{\phi}\|^2_{\brho}\,, \]
we obtain the desired estimate using \eqref{eq:ellipticest}.
\end{proof}

The issue of convergence of the discrete solution $(\bs{\rho},\bs{F})$ towards its continuous counterpart has been treated in detail in \cite{lavenant2019unconditional} for a general class of discretizations. These include the finite volume schemes considered here, in the case where the two domain decompositions coincide so that $\mathcal{I}$ is the identity operator (see remark \ref{rmk:Iidentity}). For this case, one has that the discrete denstiy $\bs{\rho}$ can be lifted to a measure on $[0,1]\times \Omega$ converging weakly to the exact optimal transport interpolation with mesh refinement. 

It is not difficult to show that the second point of theorem \ref{th:action} implies a similar convergence result, for smooth positive solutions, also when the two discretizations of the domain do not coincide (e.g., this is a direct consequence of theorem 2.18 in \cite{lavenant2019unconditional}). Besides this weak convergence result, theorem \ref{th:action} also implies the following quantitative estimate for the convergence of the potential, although in a norm dependent on the discrete solution itself.

\begin{theorem}[Convergence of the potential]\label{th:convpot}
Suppose that $\phi: [0,1]\times \Omega\rightarrow \mathbb{R}$ is an optimal potential for the dual Wasserstein problem from $\rho^{in}$ to $\rho^f$ and that $\rho:[0,1]\times \Omega \rightarrow [0,+\infty)$ is the associated interpolation. 
Let $(\tilde{\bs{\rho}},\tilde{\bs{\phi}})$ be the discrete solution associated with the boundary conditions $\brho^{in} \coloneqq\Pi_{\coarse{\mc{T}}} \rho^{in}$ and $\brho^f\coloneqq \Pi_{\coarse{\mc{T}}} \rho^f$.
If $\phi \in C^{1,1}([0,1]\times \Omega)$ and $\rho,\partial_t \rho\in L^{\infty}([0,1],C^{0,1}(\Omega))$, with $\rho\geq \varepsilon >0$, there exists a constant $C>0$ depending only on $\rho$, $\phi$, $\varepsilon$, $\zeta$ and $\Omega$, such that
\[
 \| \nabla_\Sigma \tilde{\bs{\phi}} - \nabla_\Sigma \overline{\Pi}_{\Tc} \phi\|^2_{\tilde{\bs{\rho}}} \leq C(h+\tau +\eta_h) \,,
\]
for $\eta_h$ defined as in \eqref{eq:asyiso}.
\end{theorem}

\begin{proof}
Consider the quantity 
\begin{equation}\label{eq:modenergy}
\mathcal{E}_{N,\Tc}(\tilde{\bs{\rho}},\tilde{\bs{\phi}}|\phi) \coloneqq \frac{1}{2} \| \nabla_\Sigma \tilde{\bs{\phi}} - \nabla_\Sigma \overline{\Pi}_{\Tc} \phi\|^2_{\tilde{\bs{\rho}}} \,.
\end{equation}
Expanding the square in \eqref{eq:modenergy} we obtain
\begin{equation}\label{eq:energydec}
\mathcal{E}_{N,\Tc}(\tilde{\bs{\rho}},\tilde{\bs{\phi}}|\phi) =  \mathcal{B}_{N,\mc{T}} (\tilde{\bs{\rho}},\tilde{\bs{F}}) + \frac{1}{2} \| \nabla_\Sigma  \overline{\Pi}_{\mc{T}} \phi\|^2_{\tilde{\bs{\rho}}} -  \langle   \nabla_\Sigma \tilde{\bs{\phi}},  \nabla_\Sigma  \overline{\Pi}_{\mc{T}} \phi\rangle_{\tilde{\bs{\rho}}}\,,
\end{equation}
where $\tilde{\bs{F}}$ is given by equation \eqref{eq:fluxes}.
The second term in \eqref{eq:energydec} can be written as follows
\begin{equation}
\begin{aligned}\label{eq:firstterm}
\frac{1}{2} \| \nabla_\Sigma  \overline{\Pi}_{\mc{T}} \phi\|^2_{\tilde{\bs{\rho}}}& = \frac{1}{2} \sum_{k=1}^{N+1} \int_{t^{k-1}}^{t^k}  \langle  \mathcal{L}_\Sigma^* |\nabla_\Sigma  {\Pi}_{\mc{T}} \phi(s,\cdot)|^2 - \Pi_\mc{T} |\nabla \phi (s,\cdot)|^2 , \Ic( \frac{\tilde{\bs{\rho}}^k + \tilde{\bs{\rho}}^{k-1}}{2} )\rangle_{\Tc} \, \ed s \\
& \quad -  \sum_{k=1}^{N+1} \int_{t^{k-1}}^{t^k}  \langle \Pi_\mc{T} \partial_t \phi (s,\cdot) , \Ic( \frac{\tilde{\bs{\rho}}^k + \tilde{\bs{\rho}}^{k-1}}{2} )\rangle_{\Tc} \,\ed s \\
& = I_1  -  \sum_{k=1}^{N+1} \langle \Pi_\mc{T}  {\phi} ({t}^{k},\cdot) - \Pi_\mc{T}{\phi} ({t}^{k-1},\cdot)  , \Ic(\frac{\tilde{\bs{\rho}}^k + \tilde{\bs{\rho}}^{k-1}}{2}) \rangle_{\Tc} \,.
\end{aligned}
\end{equation}
The third term in \eqref{eq:energydec} can be written as follows
\begin{equation}\label{eq:secondterm}
\begin{aligned}
- \langle   \nabla_\Sigma \tilde{\bs{\phi}},  \nabla_\Sigma  \overline{\Pi}_{\mc{T}} \phi\rangle_{\tilde{\bs{\rho}}} & = \sum_{k=1}^{N+1} \int_{t^{k-1}}^{t^k} \langle \mathrm{div}_{\Tc} ((\Lc_{\Sigma}\circ \Ic)(\frac{\tilde{\brho}^k+\tilde{\brho}^{k-1}}{2}) \odot   \nabla_{\Sigma} \bphi^k), \Pi_{\mc{T}} \phi(s,\cdot) \rangle_{\Tc} \ed s \\
& = -\sum_{k=1}^{N+1} \int_{t^{k-1}}^{t^k}  \langle \Ic( \frac{\tilde{\bs{\rho}}^{k} - \tilde{\bs{\rho}}^{k-1}}{\tau}) ,  \Pi_{\mc{T}} \phi(s,\cdot) \rangle_{\Tc} \,\ed s\\
& = I_2 - \langle \Ic \tilde{\bs{\rho}}^{N+1}, \Pi_{\mc{T}} \phi(1,\cdot) \rangle_{\Tc} + \langle \Ic \tilde{\bs{\rho}}^{0}, \Pi_{\mc{T}} \phi(0,\cdot) \rangle_{\Tc} \\ & \quad + \sum_{k=1}^N \langle \Pi_\mc{T}  {\phi} ({t}^{k},\cdot) - \Pi_\mc{T}{\phi} ({t}^{k-1},\cdot)  , \Ic( \frac{\tilde{\bs{\rho}}^k + \tilde{\bs{\rho}}^{k-1}}{2}) \rangle_{\Tc} \,,
\end{aligned}
\end{equation}
where
\[
I_2 \coloneqq  \sum_{k=1}^N \int_{t^{k-1}}^{t^{k}} \langle \Pi_\mc{T}\partial_t \phi(s,\cdot) -  \Pi_\mc{T} \frac{  {\phi} ({t}^{k+1},\cdot) - {\phi} ({t}^{k},\cdot)}{\tau}  ,\Ic \tilde{\bs{\rho}}^{k-1,k}(s)  \rangle_{\Tc} \, \ed s 
\]
and $\tilde{\bs{\rho}}^{k-1,k}(s) $ is the linear interpolation between $\tilde{\bs{\rho}}^{k-1}$ and $\tilde{\bs{\rho}}^{k}$, i.e.\
$\tilde{\bs{\rho}}^{k-1,k}(s) \coloneqq \tilde{\bs{\rho}}^{k-1}(t^{k}-s)/\tau  + \tilde{\bs{\rho}}^{k}(s - t_{k-1})/\tau$.

Adding and subtracting $W^2_2(\rho^{in},\rho^f)/2 = \int_\Omega \phi(1,\cdot) \rho^f - \int_\Omega \phi(0,\cdot) \rho^{in} $
from the right-hand side of \eqref{eq:energydec}, substituting \eqref{eq:firstterm} and \eqref{eq:secondterm}, and rearranging terms we obtain 
\begin{equation}\label{eq:energydec1}
\mathcal{E}_{N,\Tc}(\tilde{\bs{\rho}},\tilde{\bs{\phi}}|\phi) =   \frac{W^2_{N,\mc{T}} ({\bs{\rho}}^{in}, \bs{\rho}^f)}{2} - \frac{W^2_2(\rho^{in},\rho^f)}{2}   +I_1 +I_2 + I_3 \,,
\end{equation}
where
\[
\begin{aligned}
I_3\coloneqq &\int_\Omega \phi(1,\cdot) \rho^f - \int_\Omega \phi(0,\cdot) \rho^{in} - \langle\Ic^* \Pi_\mc{T} \phi(1,\cdot), \brho^{f} \rangle + \langle \Ic^* \Pi_\mc{T} \phi(0,\cdot), \brho^{in} \rangle \,,
\end{aligned}
\]
since $\brho^0 = \Pi_{\coarse{\Tc}} \rho^{in}$ and $\brho^{N+1} = \Pi_{\coarse{\Tc}} \rho^{f}$. Finally, we estimate $I_1$ and $I_3$ using lemma \ref{lem:properties}, $I_2$ using the regularity of $\phi$, and the remaining term using the second point in theorem \ref{th:action}.
\end{proof}

\begin{remark}
It is easy to construct solutions to the optimality conditions \eqref{eq:geodesic}, and therefore to problem \eqref{eq:geod}, satisfying the assumptions of theorem \ref{th:action} or \ref{th:convpot}.
In fact, given any smooth compactly-supported initial potential $\phi_0:\Omega \rightarrow \mathbb{R}$, there exists $\delta>0$ such that the map $x\mapsto T_t(x) \coloneqq x+ t \nabla \phi_0(x)$ is a diffemorphism for $t\in [0,\delta]$, and $\phi(t,\cdot) = \phi_0\circ T_t^{-1}$ is a smooth solution to the Hamilton-Jacobi equation. Moreover, given a strictly postive and smooth initial density $\rho_0$, the density $\rho(t,\cdot) = (\rho_0/ \mathrm{det}(\nabla T_t) )\circ T_t^{-1}$  solves the continuity equation with velocity $\nabla \phi(t,\cdot)$, and it is also smooth and strictly positive for $t \in [0,\delta]$. Then, the curve $t \mapsto (\rho(t\delta,\cdot), \delta \phi(t\delta,\cdot))$ solves the optimality conditions \eqref{eq:geodesic} on the time interval $[0,1]$.
On the other hand, even in the case where $\rho_0$ and $\rho_1$ are smooth and strictly positive the interpolation may not even be stricly positive as shown in \cite{santambrogio2016convexity}. 
\end{remark}

\begin{remark} The quantity $\mathcal{E}_{N,\Tc}(\tilde{\bs{\rho}},\tilde{\bs{\phi}}|\phi)$ defined in equation \eqref{eq:modenergy} is the discrete $H^1$ semi-norm of the error weighted by the discrete solution $\tilde{\brho}$. Note that this can also be seen as a discretization of the modulated energy (or relative entropy) of the kinetic energy, interpreted as a convex function of $(\rho,F)$. In section \ref{sec:numerics} we will use a similar quantity in order to evaluate numerically the convergence rate of the scheme.
\end{remark} 


	\section{Primal-dual barrier method}\label{sec:barrier}
	
	We introduce now the primal-dual barrier method, the discrete optimization technique we use to deal with the uniqueness, smoothness and positivity issues and effectively solve problem (\ref{eq:geodth}). The method consists in perturbing the discrete problem with a barrier function which forces the density to be positive. Here we show that the solutions of such perturbed problem converge to the ones of the original problem, when the perturbation vanishes, therefore justifying the use of a continuation method. Finally, we will detail the implementation of the algorithm commenting on the choice of the parameters involved.
	
	The most classical barrier function used when dealing with positivity constraints is the logarithmic barrier, $-\log{\rho}$.	 
	In order to write the perturbed problem, we first define precisely the barrier,
	 \[
	 J(x) = 
	 \begin{cases}
	 -\log(x)  &\text{if } x>0, \\
	 +\infty &\text{if } x\le 0,
	 \end{cases}
	 \]
	 so that it is convex and lower semi-continuous. 
	 We define the barrier function as ~$\Jc_{N,\Tc}(\brho)= \sum_{k=1}^{N} \tau \sum_{K\in\coarse{\Tc}} J(\rho_{\coarse{K}}^k) m_{\coarse{K}} $ and the perturbed version of problem (\ref{eq:geodth}) is therefore:
	\begin{equation}\label{eq:geodthmu}
	\inf_{(\brho,\bF)\in\Cc_{N,\Tc}} \Bc_{N,\Tc} (\brho,\bF) +\mu \Jc_{N,\Tc}(\brho) \,.
	\end{equation}
	Thanks to the strict convexity of the function $\Jc_{N,\Tc}$ on $[\mb{P}_{\coarse{\Tc}}^+\setminus \{0\}]^N$, the solution $(\brho^{\mu},\bF^{\mu})$ is now unique. Proceding as in section \ref{ssec:scheme}, $\brho^{\mu}$ can be characterized as solution to the system of optimality conditions
	\begin{equation}\label{eq:geodthmu_KKT}
	\left\{
	\begin{aligned}
	&\displaystyle \Ic (\frac{\brho^k-\brho^{k-1}}{\tau}) + \mathrm{div}_{\Tc} ((\Rc_{\Sigma}\circ \Ic)(\frac{\brho^k+\brho^{k-1}}{2}) \odot   \nabla_{\Sigma} \bphi^k) =0 , \\
	&\displaystyle \Ic^* (\frac{\bphi^{k+1}-\bphi^{k}}{\tau})+ \frac{1}{4} \Rc_{\coarse{\Tc}}[\frac{\brho^k+\brho^{k-1}}{2}] (\nabla_\Sigma\bphi^k)^2 + \Rc_{\coarse{\Tc}}[\frac{\brho^{k+1}+\brho^{k}}{2}] (\nabla_\Sigma\bphi^{k+1})^2  = -\bs{s}^k, \\
	&\brho^k \odot \bs{s}^k = \bs{\mu},
	\end{aligned}
	\right.
	\end{equation}
	where $k\in\{1,..,N+1\}$ for the continuity equation and $k\in\{1,..,N\}$ for the other conditions, and where $(\bs{\mu})_{\coarse{K}}=\mu$. The variable $\bs{s}\in[\mathbb{P}_{\coarse{\Tc}}]^N,(s^k)_{\coarse{K}} = \frac{\mu}{\rho_{\coarse{K}}^k}$, has been introduced in order to decouple the optimization in $\brho$ and $\bs{s}$,  and it highlights the connection with system (\ref{eq:geodth_KKT}).
	In particular, system \eqref{eq:geodthmu_KKT} can be seen as a perturbation of \eqref{eq:geodth_KKT}, where $\rho^k_{\coarse{K}}$ and $s^k_{\coarse{K}}=-\lambda^k_{\coarse{K}}$ are automatically forced to be positive and the orthogonality is relaxed. In this way, the solution $(\bphi^{\mu},\brho^{\mu},\bs{s}^{\mu})$ is now unique, up to an additive constant for the potential, and the problem is smooth.
	
	As it is classical in interior point methods (see, e.g., \cite{ConvOptBoyd}), if we regard $(\brho^{\mu},\bF^{\mu})$ as an approximate solution to problem \eqref{eq:geodth}, we can derive an explicit estimate on how far it is from optimality. Given a solution $(\brho,\bF)$ of the original problem, and defining $\tilde{\blambda}\in[\mathbb{P}_{\coarse{\Tc}}]^N$ by $(\tilde{\lambda}^k)_{\coarse{K}} = -\frac{\mu}{\rho_{\coarse{K}}^k}$, we have
	\begin{equation}
	\begin{aligned}
	\Bc_{N,\Tc} (\brho,\bF) &= \sup_{\bphi} \inf_{\brho\geq 0,\bF} \Lc_{N,\Tc}(\bphi,\brho,\bF)\\
	&\ge\inf_{\brho \geq 0,\bF} \Lc_{N,\Tc}({\bphi}^\mu,\brho,\bF) + \sum_{k=1}^{N} \tau \langle \tilde{\blambda}^k , \brho^k \rangle_{\coarse{\Tc}}  \\
	&= \Lc_{N,\Tc}(\bphi^\mu,\brho^{\mu},\bF^{\mu}) + \sum_{k=1}^{N} \tau \langle \tilde{\blambda}^k , (\brho^{\mu})^k \rangle_{\coarse{\Tc}}  = \Bc_{N,\Tc} (\brho^{\mu},\bF^{\mu}) - \mu \frac{N}{N+1} |\Omega| \,,
	\end{aligned}
	\end{equation}
	where we used the fact that $(\brho^{\mu},\bF^{\mu})$ is optimal for $ \Lc_{N,\Tc}({\bphi^\mu},\brho,\bF) + \sum_{k=1}^{N} \tau \langle \tilde{\blambda}^k , \brho^k \rangle_{\coarse{\Tc}}$, which can be easily verified by comparing the associated optimality conditions with \eqref{eq:geodthmu_KKT}. We have therefore
	\begin{equation}\label{eq:opt_bound}
	0\le  \Bc_{N,\Tc} (\brho^{\mu},\bF^{\mu}) - \Bc_{N,\Tc} (\brho,\bF) \le \mu \frac{N}{N+1} |\Omega| \,.
	\end{equation}
	As a consequence of (\ref{eq:opt_bound}), the smaller the parameter $\mu$, the closer the perturbed solution is to the original one.
	
	\begin{theorem}\label{thm:muconvergence}
	The solution $(\brho^{\mu},\bF^{\mu})$ of problem (\ref{eq:geodthmu}) converges up to extraction of a subsequence to $(\brho,\bF)$ solution of (\ref{eq:geodth}) for $\mu\rightarrow 0$.
	\end{theorem}
	\begin{proof}
	Consider a sequence $(\mu_n)_n \subset \mathbb{R}^+$ converging to zero and the corresponding sequence $(\brho^{\mu_n},\bF^{\mu_n})$ of solutions to problem (\ref{eq:geodthmu}).
	We first derive a bound on $(\brho^{\mu_n},\bF^{\mu_n})$, independent of $\mu$.
	The bound on $\brho^{\mu_n}$ derives easily from the conservation of mass. To obtain a bound for the momentum $\bF^{\mu_n}$, for any $\bs{b} \in [\mathbb{F}_{\Tc}]^{N+1}$ with $|b^k_\sigma| \le 1$ for all $\sigma \in \Sigma, k \in \{1,..,N+1\}$, we observe that there exists a constant $C>0$ independent of $\mu$ such that
	\begin{equation}\label{eq:Fmubound}
	\sum_{k=1}^{N+1} \tau \langle (\bF^{\mu_n})^k , \bs{b}^k \rangle_{\mathbb{F}_{\Tc}} \le \sqrt{ 2 \Bc_{N,\Tc}(\brho^{\mu_n},\bF^{\mu_n})} || \bs{b} ||_{\brho^{\mu}} \le C \,,
	\end{equation}
	where the weighted norm $ || \cdot ||_{\brho^{\mu}}$ is defined via \eqref{eq:spacetimeprod}. Note that the first inequality derives from a simple rescaling argument with the term $\brho_s \in \mb{P}^{N+1}_{\Sigma}$, given by
	\[
	\brho_s^k = \sqrt{(\Rc_{\Sigma} \circ \Ic) \frac{(\brho^{\mu_n})^k+(\brho^{\mu_n})^{k-1}}{2}} \,,
	\]
	 and applying Cauchy-Schwarz. The second one is obtained using the inequality \eqref{eq:opt_bound}.
	Taking the sup with respect to  $\bs{b}$ in (\ref{eq:Fmubound}) we obtain the bound on $\bF^{\mu_n}$.
	
	The sequence $(\brho^{\mu_n},\bF^{\mu_n})$ is bounded hence we can extract a converging subsequence (still labeled with $\mu_n$ for simplicity) $(\brho^{\mu_n},\bF^{\mu_n})\rightarrow (\brho^*,\bF^*)$. Consider $(\brho,\bF)$ minimizer of the unperturbed problem (\ref{eq:geodth}). Using inequality (\ref{eq:opt_bound}) and taking the $\liminf$ for $n \rightarrow +\infty$, we obtain $\Bc_{N,\Tc}(\brho^*,\bF^*)=\Bc_{N,\Tc}(\brho,\bF)$, hence $(\brho^*,\bF^*)$ is a minimizer for problem (\ref{eq:geodth}).
	
\end{proof}

	\begin{remark} If the solution $(\brho,\bF)$ of the discrete problem (\ref{eq:geodth}) is unique, then the entire sequence $(\brho^{\mu_n},\bF^{\mu_n})$ converges to it. In case it is not unique, for any  solution $(\brho,\bF)$
	\[
	0 \le \Bc_{N,\Tc}(\brho^{\mu_n},\bF^{\mu_n}) - \Bc_{N,\Tc} (\brho,\bF) \le \mu_n ( \Jc_{N,\Tc}(\brho) - \Jc_{N,\Tc} (\brho^{\mu_n}) ) \,,
	\]
	and therefore $(\brho^{\mu_n},\bF^{\mu_n})$ converges up to subsequence to a solution $(\brho^*,\bF^*)$ with minimal $\Jc_{N,\Tc}$. In case the solution $\brho^*$ is strictly positive everywhere, the whole sequence $(\brho^{\mu_n},\bF^{\mu_n})$ converges again.
	\end{remark}
	
	The strict positivity derives automatically from the definition of the barrier function, which attains the value $+\infty$ in zero. As a consequence, for every value of $\mu>0$ the objective function $\Bc_{N,\Tc}(\brho,\bF)+\mu \Jc_{N,\Tc}(\brho)$ is smooth in a neighborhood of the solution $(\brho^{\mu},\bF^{\mu})$, ensuring a good behavior of the Newton scheme for the solution of the system of equations (\ref{eq:geodthmu_KKT}). It is possible to derive a quantitative bound for the positivity of $\brho^{\mu}$ as follows.
	
	\begin{proposition}\label{prop:positivity}
	There exists a constant $C>0$ independent of $\mu$ such that the density $\brho^{\mu}$ solution to problem (\ref{eq:geodthmu}) satisfies the following bound:
	\begin{equation}\label{eq:boundmu}
	(\rho^{\mu})_{\coarse{K}}^k \ge C \mu, \quad \forall \coarse{K} \in \coarse{\Tc}, \, \forall k \,.
	\end{equation}
	\end{proposition}
	\begin{proof}
	Consider the solution $(\brho^{\mu},\bF^{\mu})$ to (\ref{eq:geodthmu}). We define the constant density $\bs{c} \in [\mb{P}^+_{\coarse{\Tc}}]^N, c^k_{\coarse{K}} = ({\sum_{K\in \Tc} m_{\coarse{K}}})^{-1}$. It can be easily checked that $\bs{c}$ is solution to
	\[
	\min_{\brho \in [\mb{P}_{\coarse{\Tc}}]^N} \Jc_{N,\Tc}(\brho) \quad \text{such that} \sum_{\coarse{K}\in\coarse{\Tc}} \rho_{\coarse{K}}^k m_{\coarse{K}} = 1, \forall k\,.
	\]
	From now on, with a slight abuse of notation, we consider $\bs{c}$ to be complemented with the boundary conditions $\brho^{in},\brho^f$. Thanks to the surjectivity of the divergence operator (to the space of discrete functions in $[\mathbb{P}_{\Tc}]^{N+1}$ with zero mean), we can find the momentum $\bF^c$, with minimal $|| \cdot ||_{\bs{c}}$ norm (defined via equation \eqref{eq:spacetimeprod}), such that $(\bs{c},\bF^c)\in\Cc_{N,\Tc}$.
	Taking the admissible competitor $(\hat{\brho},\hat{\bF}) = (\epsilon\bs{c} +(1-\epsilon) \brho^{\mu},\epsilon\bF^c +(1-\epsilon) \bF^{\mu}), \, \epsilon \in [0,1]$, for problem (\ref{eq:geodthmu}), it holds
	\begin{equation}\label{eq:ineqb1}
	\mu \left( \Jc_{N,\Tc}(\brho^{\mu}) - \Jc_{N,\Tc}(\hat{\brho}) \right) \le \Bc_{N,\Tc}(\hat{\brho},\hat{\bF})  -\Bc_{N,\Tc}(\brho^{\mu},\bF^{\mu}) \,.
	\end{equation}
	The right hand side of (\ref{eq:ineqb1}) is bounded: indeed, by convexity of $\Bc_{N,\Tc}$, it holds
	\begin{equation}\label{eq:boundepsA}
	\begin{aligned}
	\Bc_{N,\Tc}(\hat{\brho},\hat{\bF})  -\Bc_{N,\Tc}(\brho^{\mu},\bF^{\mu}) &\le  \epsilon \Bc_{N,\Tc}(\brho^c,\bF^c) +(1-\epsilon) \Bc_{N,\Tc}(\brho^{\mu},\bF^{\mu}) - \Bc_{N,\Tc}(\brho^{\mu},\bF^{\mu}) \\
	&\le C \epsilon \,.
	\end{aligned} 
	\end{equation}
	The left hand side of (\ref{eq:ineqb1}) can be bounded from below thanks to the convexity of $\Jc_{N,\Tc}$, by the following quantity
	\[
	\mu \sum_{k=1}^{N}\sum_{\coarse{K}\in\coarse{\Tc}} J'(\hat{\rho}_{\coarse{K}}^k) ((\rho^{\mu})_{\coarse{K}}^k-\hat{\rho}_{\coarse{K}}^k) m_{\coarse{K}} \tau =  \mu \epsilon \sum_{k=1}^{N}\sum_{\coarse{K}\in\coarse{\Tc}} J'(\hat{\rho}_{\coarse{K}}^k) ((\rho^{\mu})_{\coarse{K}}^k-c_{\coarse{K}}^k)) m_{\coarse{K}} \tau \,.
	\]
	Hence, we obtain
	\begin{equation}\label{eq:ineqb2}
  \mu \epsilon \sum_{k=1}^{N}\sum_{\coarse{K}\in\coarse{\Tc}} J'(\hat{\rho}_{\coarse{K}}^k) ((\rho^{\mu})_{\coarse{K}}^k-c_{\coarse{K}}^k)) m_{\coarse{K}} \tau \le C \epsilon \,.
	\end{equation}
	Simplifying $\epsilon$ in (\ref{eq:ineqb2}) and taking the limit for $\epsilon \rightarrow 0$, we obtain
	\[
	\sum_{k=1}^{N}\sum_{\coarse{K}\in\coarse{\Tc}} \left(\frac{c_{\coarse{K}}^k}{(\rho^{\mu})_{\coarse{K}}^k} -1\right)m_{\coarse{K}} \tau \le \frac{C}{\mu} \implies \min_{\coarse{K}} (m_{\coarse{K}}) \tau \sum_{k=1}^{N}\sum_{\coarse{K}\in\coarse{\Tc}} \frac{c_{\coarse{K}}^k}{(\rho^{\mu})_{\coarse{K}}^k} \le  \frac{C}{\mu} + |\Omega| T, 
	\]
	which implies the result.
\end{proof}	
	
	By theorem \ref{thm:muconvergence} the solution of problem (\ref{eq:geodthmu}) provides an approximation to a solution $(\bphi,\brho)$ to problem \eqref{eq:geodth_KKT}, although the smaller the parameter the more difficult it is to solve the problem using a Newton method.
	The idea is then to use a continuation method, that is construct a sequence of solutions to problem \eqref{eq:geodthmu_KKT} for a sequence of coefficients $\mu$ decreasing to zero, using each time the solution at the previous step as starting point for the Newton scheme.
	The resulting algorithm in shown in Algorithm \ref{alg:IPM}. We denote by $\theta$  the rate of decay for $\mu$;  by $\varepsilon_0$ and $\varepsilon_{\mu}$  the tolerances for the solution to \eqref{eq:geodth_KKT} and \eqref{eq:geodthmu_KKT}, respectively; and by $\delta_0$ and $\delta_{\mu}$ the error in the convergence towards solutions of the original and perturbed problem. The parameter $\delta_{\mu}$ can be taken to be a norm of the residual of the system of equations (\ref{eq:geodthmu_KKT}) or of the Newton step $\bs{d}$. Concerning $\delta_0$, it is either possible to choose a norm of the residual of the system of equations (\ref{eq:geodth_KKT}) or $\delta_0 = \mu \frac{N}{N+1} |\Omega|$, by virtue of \eqref{eq:opt_bound}, whether the proximity to the minimizer or to the minimum is preferred.
		
	\begin{algorithm}
		\SetAlgoLined
		Given the starting point $(\bphi_0,\brho_0,\bs{s}_0)$ and the parameters $\mu_0>0,\theta\in(0,1),\varepsilon_0>0$ \; 
		\While{$\delta_0>\varepsilon_0$}{
			$\mu = \theta \mu$ \;
			\While{$\delta_{\mu}>\varepsilon_{\mu}$}{
			compute Newton direction $\bs{d}$ for \eqref{eq:geodthmu_KKT}\;
			compute $\alpha\in(0,1]$ such that $\brho+\alpha \boldsymbol{d}_{\brho}>0$ and $\bs{s}+\alpha \bs{d}_{\bs{s}}>0$\;
			update: $(\bphi,\brho,\bs{s}) = (\bphi,\brho,\bs{s}) + \alpha (\bs{d}_{\bphi},\bs{d}_{\brho},\bs{d}_{\bs{s}})$ \;
			\If{$n>n_{max}$ \textnormal{\textbf{or}} $\alpha < \alpha_{min}$}{
				increase $\mu$ and repeat \;
			}
			}
		}
		\caption{}
		\label{alg:IPM}
	\end{algorithm}
	
	Since any intermediate solution for $\mu\neq0$ is not of interest, a very common approach in interior point methods is to set a relatively big tolerance $\varepsilon_{\mu}$, or even to do just one Newton step per value of $\mu$.
	Nonetheless, a small tolerance $\varepsilon_{\mu}$ avoids the density to get accidentally too close to the boundary of the feasibility domain, i.e. too close to zero, which would imply a drop in the regularity of the specific problem at hand. For this reason we consider $\varepsilon_{\mu} = \varepsilon_0$.
		
	A linesearch technique is typically employed in order to ensure global convergence of the Newton scheme. However, in many cases it leads to a non negligible cost by forcing the Newton scheme to do several steps before reaching convergence. Instead of modifying the step size $\alpha$, we adaptively control $\theta$ in order to force the convergence. The Newton scheme is repeated with an increased $\theta$ (i.e.\ with an increased $\mu$) if it is not able to converge in $n_{max}$ steps. The step size $\alpha$ is chosen just to ensure that $\brho$ and $\bs{s}$ do not become negative. Again, the Newton scheme is repeated if $\alpha$ needs to be smaller than $\alpha_{min}$. In particular, taking $\alpha_{min} = 1$ one only allows full Newton steps.
	
	There exist of course several optimization solvers that could tackle the solution of problem \eqref{eq:geodth}, most of which are usually based on interior point strategies, especially for large scales. Nevertheless, the specificity of the problem at hand, its non-linearity of course but more importantly its lack of smoothness, led us to develop our own solver, in order to better handle it. Moreover, the solution of the sequence of linear systems requires an ad-hoc strategy, as mentioned in sections \ref{sec:numerics}-\ref{sec:perspectives}.
	Finally, we remark that in the particular case of the linear reconstruction, the corresponding dual problem in \eqref{eq:dualproblem} can be cast in the form of a second-order cone program, which can be solved again using an interior point method in polynomial time. This does not apply to the case of the harmonic reconstruction (or more general reconstructions) for which the dual problem has a more complex structure.


\section{Numerical results}\label{sec:numerics}

	In this section we assess the performance of the scheme using several two-dimensional numerical tests.
	In particular, we demonstrate the numerical implications of enriching the space of discrete potential, both from a qualitative and quantitative point of view. As already noted in remark \ref{rmk:Iidentity}, considering the two subdivisions of the domain to be the same and taking $\Ic$ to be the identity operator, we recover the discretization presented in \cite{lavenant2019unconditional}. We will refer to this case as the non-enriched scheme.
	Needless to say, the greater is the richness of the space of discrete potentials the higher is the computational complexity.
	
	For the construction of the enriched scheme we use the nested meshes described in section \ref{ssec:mesh}. In particular, the coarse mesh is given by a regular triangulation of the domain with only acute angles. Here, we will use the first family of grids provided in \cite{fvca5}, which discretize the domain $\Omega = [0,1]^2$.
	
	The code is implemented in MATLAB and is available online\footnote{\href{https://github.com/gptod/OT-FV}{https://github.com/gptod/OT-FV}}. In particular, we exploit the built-in MATLAB direct solver to solve the sequence of linear systems generated by Algorithm \ref{alg:IPM}.
	For $\mu\rightarrow 0$ the Jacobian matrix becomes ill-conditioned and the computation time, along with the memory consumption, rapidly increases for this solver.
	Using an iterative method could be extremely beneficial in this sense.
	However, the design of effective preconditioners is a delicate issue and should take into account the structure of the problem at hand (see, e.g., the general survey \cite{BenziSurvey}). Therefore, we do not explore the use of such techniques in this article.
	We calibrated Algorithm \ref{alg:IPM} with the following parameters: $\theta=0.2$, $\alpha_{min}=0.1$, $\epsilon_0=10^{-6}$ ($\epsilon_0=10^{-8}$ for the convergence tests), $\mu_0=1, \bphi_0=\bs{0},\brho_0=\bs{c}$ ($\bs{c}$ defined as in \ref{prop:positivity}) and $\bs{s}_0$ satisfying $\brho_0 \odot \bs{s}_0 = \bs{\mu}_0$. In all the simulations performed in this section, but also more generally, the algorithm proved to be extremely robust under this configuration. The Newton scheme rarely reaches a breakdown and, in case this happens, the adaptive strategy on the parameter $\theta$ overcomes the issue. Notice only that for complex simulations the value $\mu_0$ may be increased to ease the start of the Newton scheme.
	Finally, we stress that all our results are presented in their piecewise-constant form on the grid, without any kind of interpolation.


\subsection{Oscillations}\label{ssec:oscillations}

	In this section we show that the discrete density obtained by using the non-enriched scheme can be very oscillatory. We observed numerically that the oscillations are more severe in cases where there is high compression of mass, i.e. when the corresponding continuous velocity field is not divergence free, and also more persistent with refinement (this is also confirmed by the convergence tests shown below in section \ref{ssec:testconv}).
	On the other hand, this type of instability can be prevented using the enriched scheme, which eliminates the oscillations almost entirely.
		
	In order to illustrate this phenomenon, we consider the interpolation between the two densities
	\[
	\rho^{in}(x,y) =  \cos\left( 2\pi\left|\x-\x_0\right| \right) +\frac{3}{2} \,, \quad \rho^f(x,y) = -\cos\left( 2\pi\left|\x-\x_0\right| \right)+\frac{3}{2} \,,
	\]
	where $\x = (x,y)$ and $\x_0=(\frac{1}{2},\frac{1}{2})$. For $\coarse{h}=0.0625$ and $\#\coarse{\Tc}=896$, and for a number $N+1=8$ of time steps, we compute the approximate Wasserstein interpolation between $\brho^{in} = \left(\rho^{in}(\x_K)\right)_{\coarse{K}\in\coarse{\Tc}}$ and $\brho^{f} = \left(\rho^{f}(\x_K)\right)_{\coarse{K}\in\coarse{\Tc}}$, by solving problem (\ref{eq:geodth}), in four different ways: with the enriched and the non-enriched schemes, both with linear and harmonic reconstruction. The results are shown in figure \ref{fig:cosmid}.
	The non-enriched scheme with linear reconstruction exhibits severe oscillations which disappear using the enriched one. Oscillations are evident also using the harmonic reconstruction. The enriched scheme with harmonic reconstruction provides the smoothest solution. 
	
	\begin{figure}
		\centering
		\includegraphics[trim={2.5cm 1cm 1.9cm 0.5cm},clip,width=0.4\textwidth]{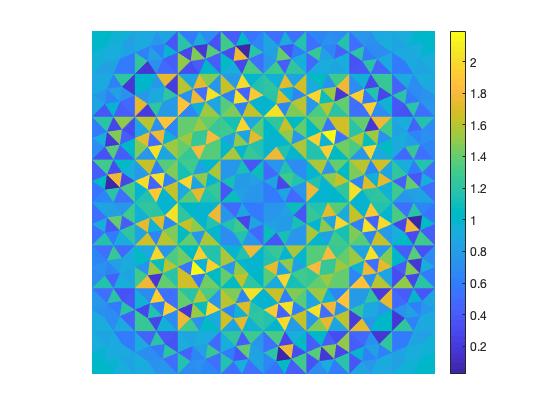}
		\includegraphics[trim={2.5cm 1cm 1.9cm 0.5cm},clip,width=0.4\textwidth]{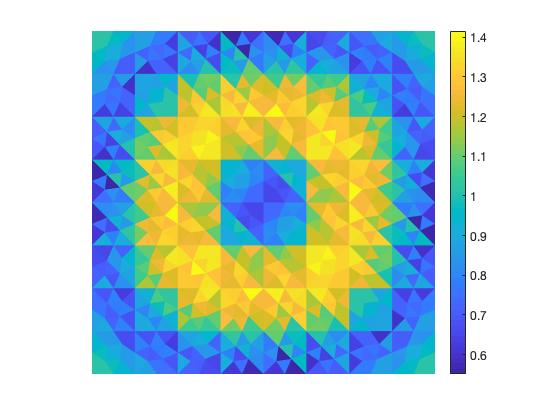} \\
		\includegraphics[trim={2.5cm 1cm 1.9cm 0.5cm},clip,width=0.4\textwidth]{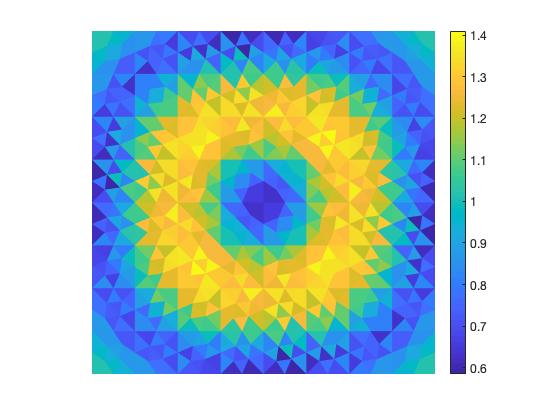}
		\includegraphics[trim={2.5cm 1cm 1.9cm 0.5cm},clip,width=0.4\textwidth]{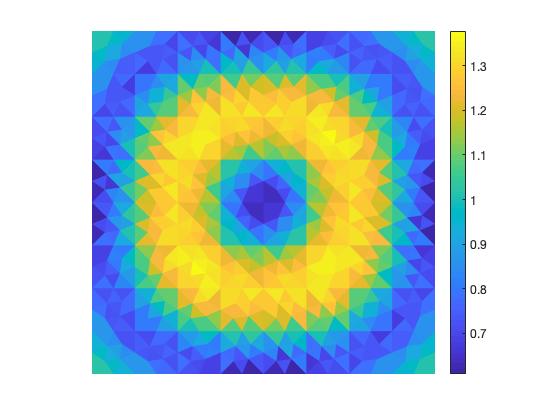}
		\caption{Midpoint between two sinusoidal functions. Non-enriched scheme in the top row, enriched scheme in the bottom one. Linear reconstruction on the left, harmonic reconstruction on the right.}
		\label{fig:cosmid}
	\end{figure}

	It is worth mentioning that the non-enriched scheme does not exhibit oscillations for rectangular cartesian grids. Indeed, oscillations do not appear either in other works based on finite differences \cite{carrillo2019primal,papadakis2014optimal,osher2018computations}, which coincide with finite volumes on such simple grids.


	\subsection{Convergence test}\label{ssec:testconv}
	
	We now quantify numerically the convergence rate for the potential, the Wasserstein distance and the density, by considering specific smooth solutions $(\phi,\rho)$ to (\ref{eq:geod}) with compact support, and with smooth initial and final densities $\rho^{in}$ and $\rho^f$. 
Note, however, that the convergence results of section \ref{sec:convergence} are less general, since they require strictly positive densities, and only apply to the linear reconstruction.
	
	We compute the solutions to problem (\ref{eq:geodth}), with $\brho^{in}=(\rho^{in}(\x_{\coarse{K}}))_{\coarse{K}\in\coarse{\Tc}}, \brho^{f}=(\rho^{f}(\x_\coarse{K}))_{\coarse{K}\in\coarse{\Tc}}$, on a sequence of admissible meshes $\left(\coarse{\mathcal{T}}, \coarse{\overline{\Sigma}}, {(\mathbf{x}_{\coarse{K}})}_{\coarse{K}\in\mathcal{T}}\right)$, and with an increasing number of time steps.
	We consider four type of errors: the error on the distance, the $L^1$ error on the density curve, the weighted $L^2$ error on the potential and on its gradient on the whole trajectory. We define a discrete potential $\bphi\in [\mathbb{P}_{\Tc}]^{N+1}$ by sampling the continuous solution, i.e.\ $\bphi^k_K = \phi(t^{k-1}+\frac{\tau}{2},\x_K)$, for $k\in\{1,..,N+1\}$, and similarly for the density we introduce $\brho\in[\mathbb{P}_{\Tc}]^{N+1}$, with $\brho_K^k = \rho(t^{k-1}+\frac{\tau}{2},\x_K)$, for $k\in\{1,..,N+1\}$. Given the discrete solution $(\tilde{\bphi},\tilde{\brho})$, the four errors are then computed as follows:
	\begin{alignat*}{2}
	&\epsilon_{W_2} = | W(\rho^{in},\rho^{f}) - W_{N,\Tc}(\brho^{in}, \brho^f) | \,, \hspace{2em}
	& &\epsilon_{\phi} = \sum_{k=1}^{N+1} \tau \langle (\tilde{\phi}^k_K -\phi^k_K)^2 , \Ic (\frac{\tilde{\brho}^k+\tilde{\brho}^{k-1}}{2}) \rangle_{\Tc}\,, \\
	&\epsilon_{\nabla \phi} =  \| \nabla_\Sigma  \tilde{\bphi} - \nabla_\Sigma \bphi \|_{\tilde{\brho}}\,, \quad
	& &\epsilon_{\rho} = \sum_{k=1}^{N+1}\tau \sum_{\coarse{K}\in\coarse{\Tc}} | \brho^k_K - \frac{\tilde{\brho}^{k}_K +\tilde{\brho}^{k-1}_K}{2}| m_{\coarse{K}} \,,
	\end{alignat*}
	where the weighted (semi-)norm $ || \cdot ||_{\tilde{\brho}}$ is defined via \eqref{eq:spacetimeprod}. 
	
	We first consider the simple case of a pure translation. We consider the optimal transport problem between the two following densities:
	\[
	\begin{aligned}
	&\rho^{in}(x,y) = \left( 1+\cos \left( \frac{10^2\pi}{3^2} \left|\x-\x_1\right|^2 \right) \right) \boldsymbol{1}_{\left|\x-\x_1\right|\leq \frac{3}{10}} \,, \\
	&\rho^{f}(x,y) = \left(1+\cos \left( \frac{10^2\pi}{3^2} \left|\x-\x_2\right|^2 \right) \right) \boldsymbol{1}_{\left|\x-\x_2\right|\leq \frac{3}{10}} \,,
	\end{aligned}
	\]
	where $\x_1=(\frac{3}{10},\frac{3}{10}), \x_2=(\frac{7}{10},\frac{7}{10})$.
	The density interpolation and the potential are simply given by
	\[
	\begin{aligned}
	&\rho(t,x,y) = \left( 1+\cos \left( \frac{10^2\pi}{3^2} \left|\x-\x_t\right|^2 \right) \right) \boldsymbol{1}_{\left|\x-\x_t\right|\leq \frac{3}{10}} \,, \\
	&\phi(t,x,y) = \frac{2}{5}x+\frac{2}{5}y-\frac{4}{25}t \,,
	\end{aligned}
	\]
	where $\x_t=(1-t)\x_1+t\x_2=(\frac{3}{10}+\frac{2}{5}t,\frac{3}{10}+\frac{2}{5}t)$, and the Wasserstein distance is $W_2(\rho^{in},\rho^{f}) =\frac{2\sqrt{2}}{5}$. Note in particular that the associated velocity field is constant in space. The errors defined above and the respective rates of convergence are shown in table \ref{tab:translation}. In this case, all the considered errors converge with a rate of at least one for both the enriched and non-enriched scheme and both type of reconstructions.
		
	\begin{table}
	\caption{Convergence test on the translation.}
	\label{tab:translation}
		\begin{tabular}{cccccccccc}
			\toprule
			\toprule
			$\coarse{h}$ & $N$ & $\epsilon_{W_2}$ & rate & $\epsilon_{\phi}$ & rate & $\epsilon_{\nabla \phi}$ & rate & $\epsilon_{\rho}$ & rate \\
			\toprule
			\toprule
			\multicolumn{10}{c}{Non-enriched scheme with linear reconstruction} \\
			\midrule
			0.250 &  1 & 3.109e-02 & / & 1.802e-02 & / & 2.153e-01 & / & 6.092e-01 & / \\ 
			0.125 &  3 & 3.375e-03 & 3.204 & 4.857e-03 & 1.892 & 9.574e-02 & 1.169 & 2.779e-01 & 1.132 \\ 
			0.062 &  7 & 1.190e-03 & 1.504 & 1.442e-03 & 1.752 & 3.947e-02 & 1.278 & 1.431e-01 & 0.958 \\ 
			0.031 & 15 & 2.351e-04 & 2.339 & 4.105e-04 & 1.813 & 1.550e-02 & 1.348 & 7.115e-02 & 1.008 \\ 
			0.016 & 31 & 2.874e-05 & 3.032 & 1.086e-04 & 1.919 & 5.708e-03 & 1.442 & 3.110e-02 & 1.194 \\ 
			\toprule
			\multicolumn{10}{c}{Non-enriched scheme with harmonic reconstruction} \\
			\midrule
			0.250 &  1 & 4.897e-02 & / & 2.382e-02 & / & 1.825e-01 & / & 5.870e-01 & / \\ 
			0.125 &  3 & 9.950e-03 & 2.299 & 5.635e-03 & 2.080 & 7.503e-02 & 1.282 & 2.535e-01 & 1.211 \\ 
			0.062 &  7 & 4.009e-03 & 1.311 & 1.751e-03 & 1.686 & 3.393e-02 & 1.145 & 1.172e-01 & 1.114 \\ 
			0.031 & 15 & 1.168e-03 & 1.780 & 5.055e-04 & 1.792 & 1.433e-02 & 1.243 & 4.907e-02 & 1.256 \\ 
			0.016 & 31 & 3.074e-04 & 1.925 & 1.409e-04 & 1.843 & 6.040e-03 & 1.247 & 2.057e-02 & 1.254 \\ 
			\toprule
			\multicolumn{10}{c}{Enriched scheme with linear reconstruction} \\
			\midrule
			0.250 &  1 & 3.880e-02 & / & 2.084e-02 & / & 2.231e-01 & / & 5.774e-01 & / \\ 
			0.125 &  3 & 3.714e-03 & 3.385 & 5.129e-03 & 2.023 & 9.375e-02 & 1.251 & 2.343e-01 & 1.301 \\ 
			0.062 &  7 & 1.457e-03 & 1.350 & 1.568e-03 & 1.710 & 4.303e-02 & 1.124 & 9.481e-02 & 1.305 \\ 
			0.031 & 15 & 3.551e-04 & 2.037 & 4.391e-04 & 1.836 & 1.935e-02 & 1.153 & 3.233e-02 & 1.552 \\ 
			0.016 & 31 & 6.712e-05 & 2.403 & 1.145e-04 & 1.939 & 8.719e-03 & 1.150 & 1.228e-02 & 1.397 \\ 
			\toprule
			\multicolumn{10}{c}{Enriched scheme with harmonic reconstruction} \\
			\midrule
			0.250 &  1 & 4.512e-02 & / & 2.240e-02 & / & 1.999e-01 & / & 5.740e-01 & / \\ 
			0.125 &  3 & 6.907e-03 & 2.708 & 5.187e-03 & 2.111 & 8.270e-02 & 1.273 & 2.370e-01 & 1.276 \\ 
			0.062 &  7 & 2.852e-03 & 1.276 & 1.597e-03 & 1.699 & 3.975e-02 & 1.057 & 1.036e-01 & 1.193 \\ 
			0.031 & 15 & 8.292e-04 & 1.782 & 4.521e-04 & 1.821 & 1.857e-02 & 1.098 & 4.014e-02 & 1.369 \\ 
			0.016 & 31 & 2.116e-04 & 1.970 & 1.221e-04 & 1.889 & 8.802e-03 & 1.077 & 1.668e-02 & 1.266 \\
			\bottomrule
		\end{tabular}
	\end{table}

	We now consider a more challenging test, the optimal transport problem between the two densities
	\[
	\begin{aligned}
	&\rho^{in}(x,y) = \left(1+\cos\left(2\pi\left(x-\frac{1}{2}\right)\right)\right) \,, \\ 
	&\rho^{f}(x,y) = \frac{1}{c} \left(1+\cos\left(\frac{2\pi}{c} \left(x-\frac{1}{2}\right)\right )\right) \boldsymbol{1}_{\left|x-\frac{1}{2}\right| \le\frac{c}{2}} \,,
	\end{aligned}
	\]
	where $\rho^f$ is the compression of a factor $c$ of $\rho^{in}$. The exact expression of the density interpolation is 
	\[
	\rho(t,x,y) =  \frac{1}{t(c-1)+1} \left(1+\cos\left(\frac{2\pi}{t(c-1)+1} \left(x-\frac{1}{2}\right)\right)\right) \boldsymbol{1}_{\left|x-\frac{1}{2}\right|\le\frac{t(c-1)+1}{2}} \,,
	\]
  whereas the exact potential is
	\[
	\phi(t,x,y) = \frac{1}{2} \frac{c-1}{t(c-1)+1}\left(x-\frac{1}{2}\right)^2.
	\]
	The Wasserstein distance between the two densities is
	\[
	W_2(\rho^{in},\rho^{f}) = \sqrt{\frac{(\pi^2-6) (c-1)^2}{12\pi^2}} \,.
	\]
	The numerical results for $c=0.3$ are shown in table \ref{tab:compression}.
	Again, in all the four cases, the Wasserstein distance and the gradient of the potential converge, with the errors exhibiting at least a linear rate of convergence. However, the density does not seem to converge in the non-enriched scheme with linear reconstruction, whereas it converges in the other cases.
	
	\begin{table}
	\caption{Convergence test on the compression.}
	\label{tab:compression}
		\begin{tabular}{cccccccccc}
			\toprule
			\toprule
			$\coarse{h}$ & $N$ & $\epsilon_{W_2}$ & rate & $\epsilon_{\phi}$ & rate & $\epsilon_{\nabla \phi}$ & rate & $\epsilon_{\rho}$ & rate \\
			\toprule
			\toprule
			\multicolumn{10}{c}{Non-enriched scheme with linear reconstruction} \\
			\midrule
			0.250 &  1 & 1.653e-02 & / & 4.734e-03 & / & 6.903e-02 & / & 2.288e-01 & / \\ 
			0.125 &  3 & 1.421e-03 & 3.540 & 1.471e-03 & 1.687 & 3.301e-02 & 1.064 & 1.285e-01 & 0.832 \\ 
			0.062 &  7 & 2.978e-04 & 2.255 & 4.651e-04 & 1.661 & 1.729e-02 & 0.933 & 1.859e-01 & -0.532 \\ 
			0.031 & 15 & 4.850e-04 & -0.704 & 1.466e-04 & 1.666 & 1.038e-02 & 0.736 & 2.193e-01 & -0.238 \\ 
			0.016 & 31 & 2.030e-04 & 1.257 & 4.491e-05 & 1.706 & 6.351e-03 & 0.709 & 2.378e-01 & -0.117 \\ 
			\toprule
			\multicolumn{10}{c}{Non-enriched scheme with harmonic reconstruction} \\
			\midrule
			0.250 &  1 & 2.380e-03 & / & 2.785e-03 & / & 3.954e-02 & / & 2.666e-01 & / \\ 
			0.125 &  3 & 8.112e-03 & -1.769 & 1.403e-03 & 0.989 & 2.384e-02 & 0.730 & 7.503e-02 & 1.829 \\ 
			0.062 &  7 & 2.805e-03 & 1.532 & 4.851e-04 & 1.532 & 1.162e-02 & 1.037 & 7.046e-02 & 0.091 \\ 
			0.031 & 15 & 6.207e-04 & 2.176 & 1.242e-04 & 1.966 & 5.419e-03 & 1.100 & 4.919e-02 & 0.518 \\ 
			0.016 & 31 & 1.652e-04 & 1.910 & 3.574e-05 & 1.797 & 2.690e-03 & 1.011 & 3.393e-02 & 0.536 \\ 
			\toprule
			\multicolumn{10}{c}{Enriched scheme with linear reconstruction} \\
			\midrule
			0.250 &  1 & 1.746e-02 & / & 4.130e-03 & / & 6.212e-02 & / & 2.333e-01 & / \\ 
			0.125 &  3 & 2.093e-03 & 3.060 & 9.486e-04 & 2.122 & 2.725e-02 & 1.189 & 7.694e-02 & 1.600 \\ 
			0.062 &  7 & 2.436e-04 & 3.103 & 2.827e-04 & 1.747 & 1.274e-02 & 1.097 & 5.805e-02 & 0.406 \\ 
			0.031 & 15 & 1.538e-04 & 0.664 & 7.698e-05 & 1.876 & 5.834e-03 & 1.127 & 3.551e-02 & 0.709 \\ 
			0.016 & 31 & 5.447e-05 & 1.497 & 1.932e-05 & 1.994 & 2.751e-03 & 1.085 & 2.325e-02 & 0.611 \\
			\toprule
			\multicolumn{10}{c}{Enriched scheme with harmonic reconstruction} \\
			\midrule
			0.250 &  1 & 7.281e-03 & / & 3.069e-03 & / & 4.756e-02 & / & 2.606e-01 & / \\ 
			0.125 &  3 & 2.609e-03 & 1.480 & 7.574e-04 & 2.019 & 2.332e-02 & 1.028 & 5.786e-02 & 2.171 \\ 
			0.062 &  7 & 1.626e-03 & 0.682 & 2.984e-04 & 1.344 & 1.112e-02 & 1.069 & 4.280e-02 & 0.435 \\ 
			0.031 & 15 & 2.752e-04 & 2.563 & 7.551e-05 & 1.983 & 5.378e-03 & 1.048 & 2.409e-02 & 0.829 \\ 
			0.016 & 31 & 6.788e-05 & 2.020 & 2.166e-05 & 1.802 & 2.700e-03 & 0.994 & 1.537e-02 & 0.648 \\ 
			\bottomrule
		\end{tabular}
	\end{table}
	
	It is noticeable from the convergence tests we performed how in the case of a pure translation the instability tends to disappear with refinement, whereas with compression this depends on the reconstruction used: the harmonic reconstruction seems to prevent the issue, the linear one does not. Our strategy of enriching the discrete space of potentials alleviates the problem and enables to recover the convergence of the density.
	
	
	\subsection{Geodesic}
	
	To conclude, we consider the transport problem between a cross distributed density and its rotation by $45$ degrees. We compute the discrete solution with the enriched scheme, using the harmonic reconstruction, with $\coarse{h}=0.0156, \#\coarse{\Tc}=14336$ and $N+1=32$ time steps. The approximate density interpolation is displayed in figure \ref{fig:crossinterp}: as expected, each branch of the cross splits symmetrically in two parts which move towards the two opposite branches of the rotated cross.
	
	\begin{figure}
		\centering
		\begin{minipage}{0.89\textwidth}
			\centering
			\includegraphics[trim={8cm 2.5cm 6.5cm 1.5cm},clip,width=0.30\textwidth]{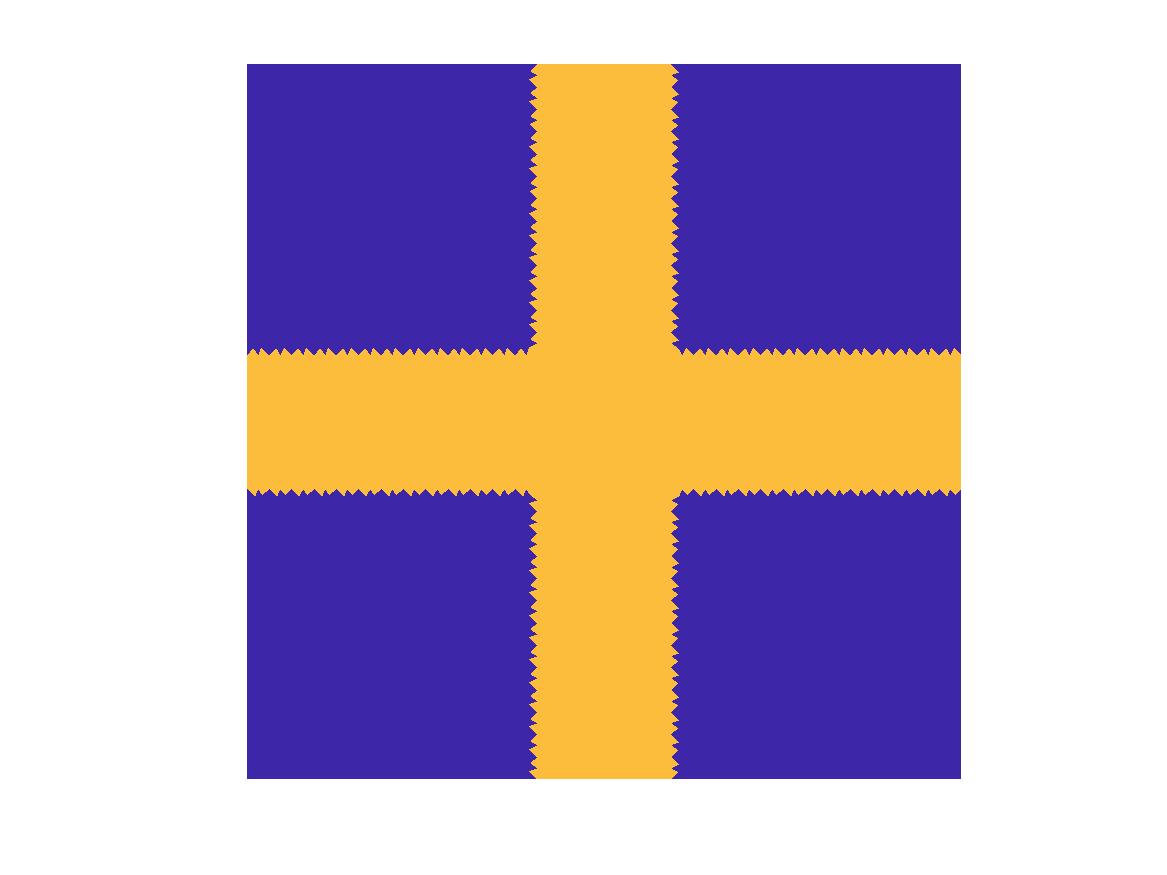} 
			\includegraphics[trim={8cm 2.5cm 6.5cm 1.5cm},clip,width=0.30\textwidth]{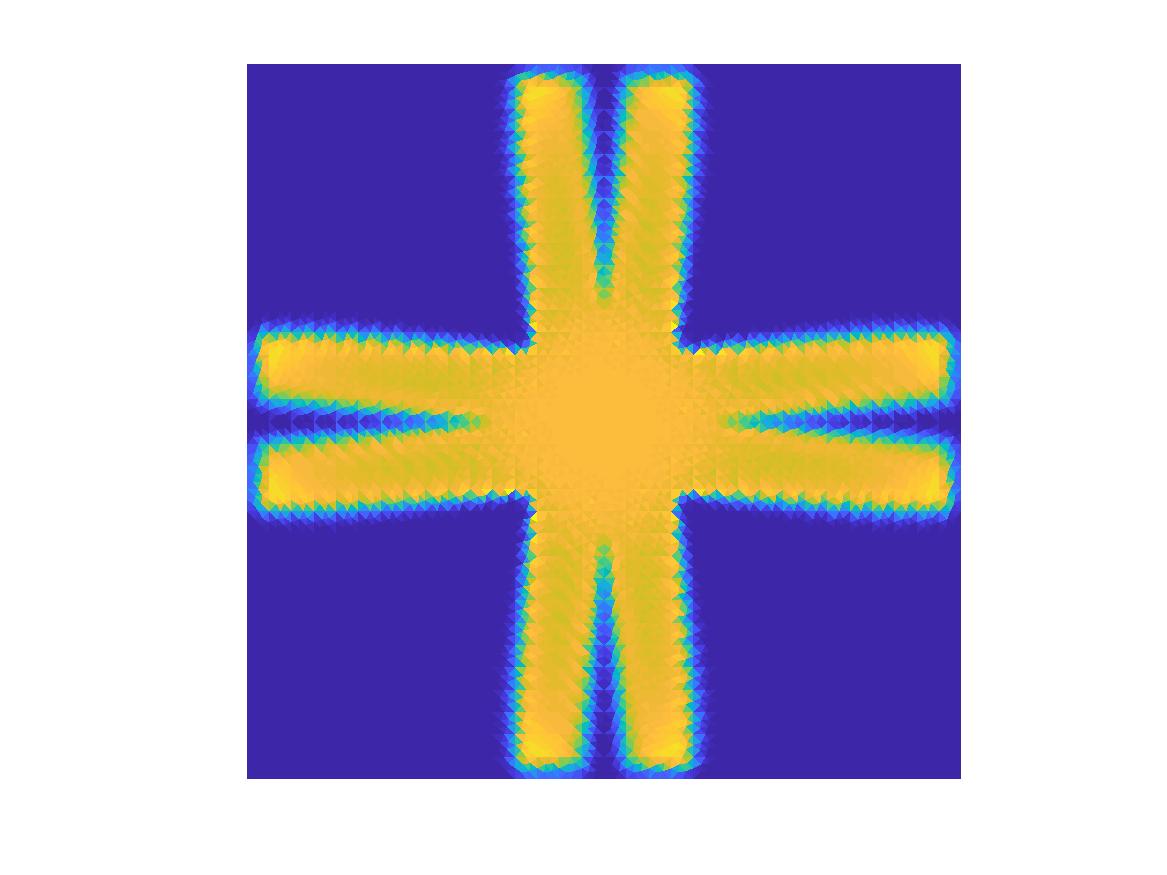} 
			\includegraphics[trim={8cm 2.5cm 6.5cm 1.5cm},clip,width=0.30\textwidth]{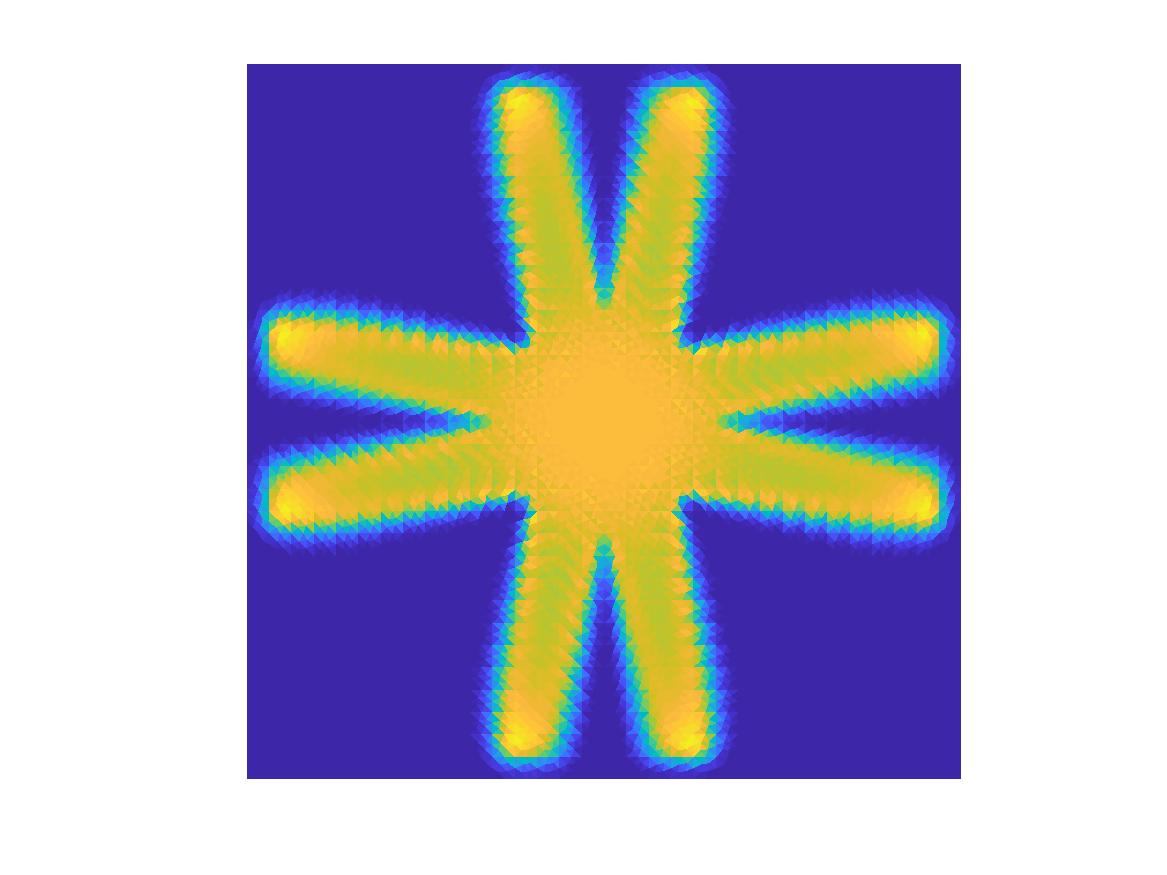} \\
			\includegraphics[trim={8cm 2.5cm 6.5cm 1.5cm},clip,width=0.30\textwidth]{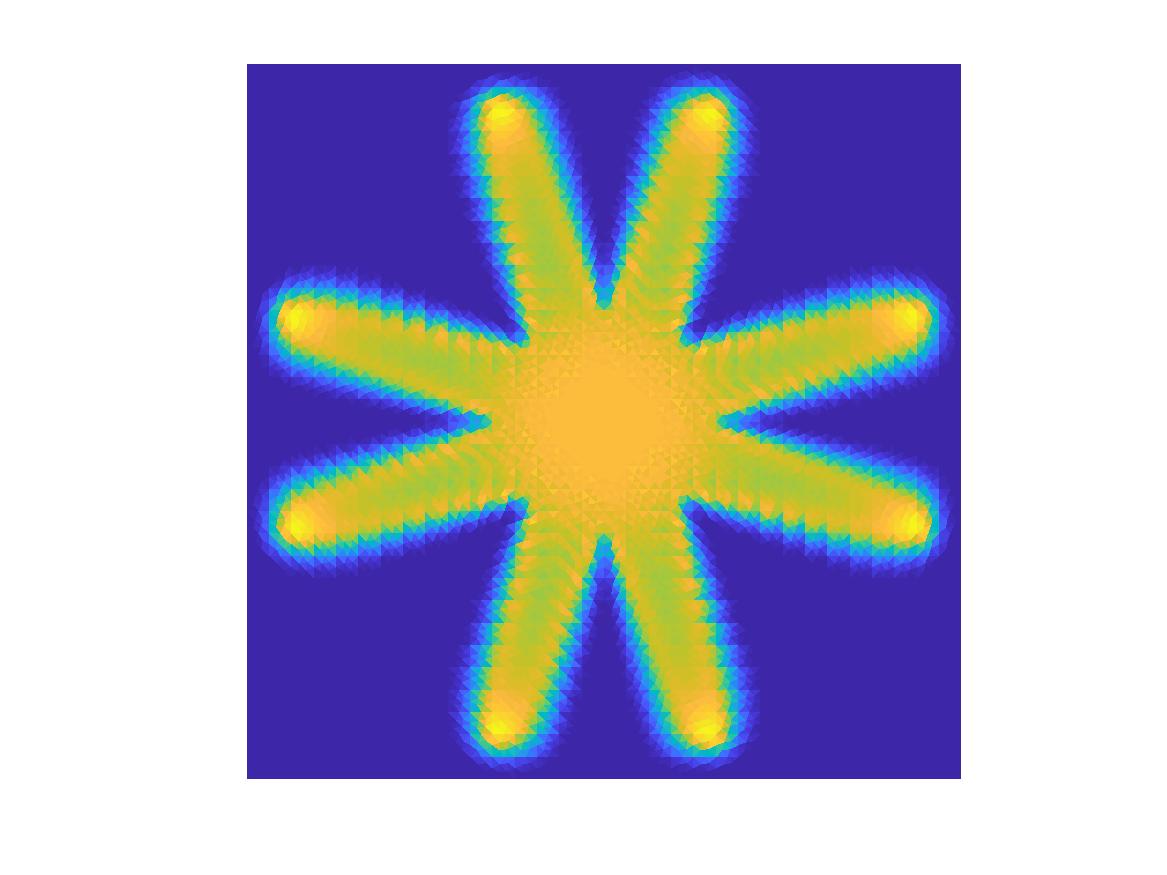}
			\includegraphics[trim={8cm 2.5cm 6.5cm 1.5cm},clip,width=0.30\textwidth]{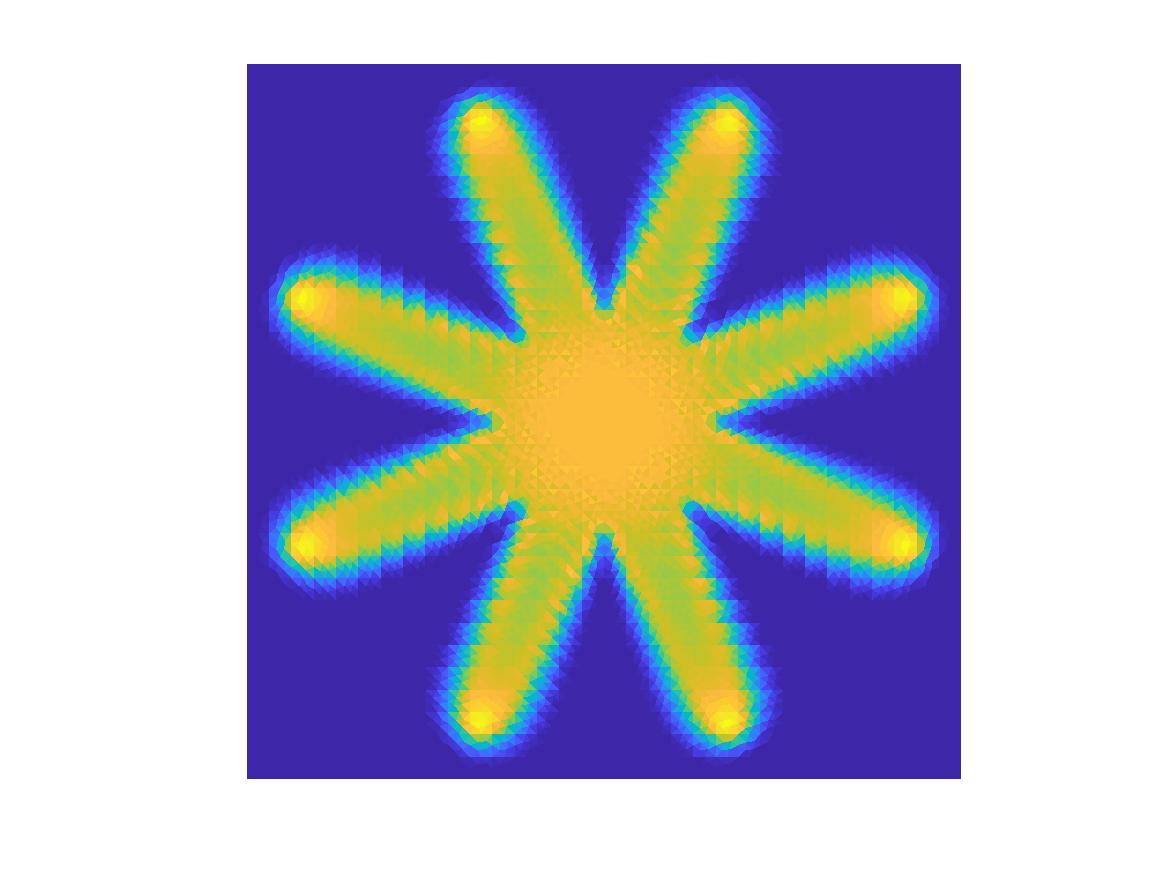}
			\includegraphics[trim={8cm 2.5cm 6.5cm 1.5cm},clip,width=0.30\textwidth]{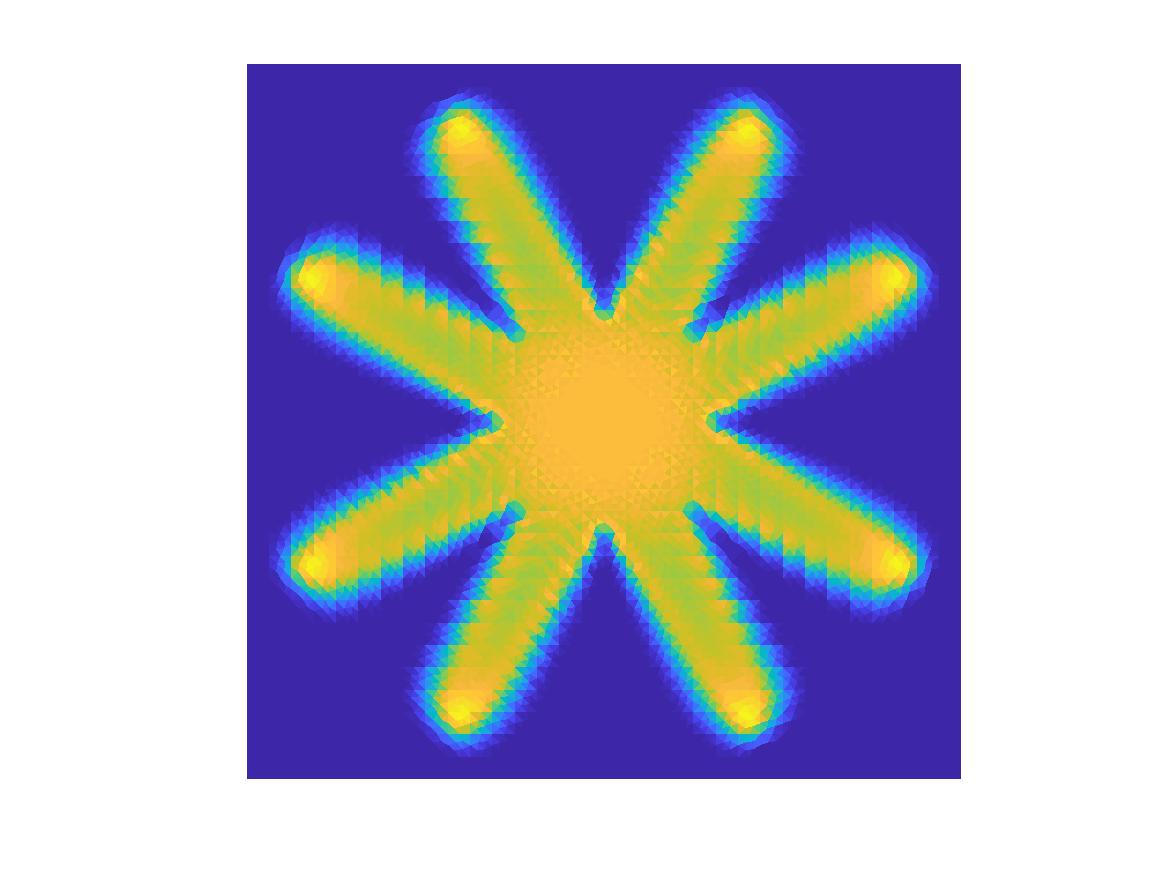} \\
			\includegraphics[trim={8cm 2.5cm 6.5cm 1.5cm},clip,width=0.30\textwidth]{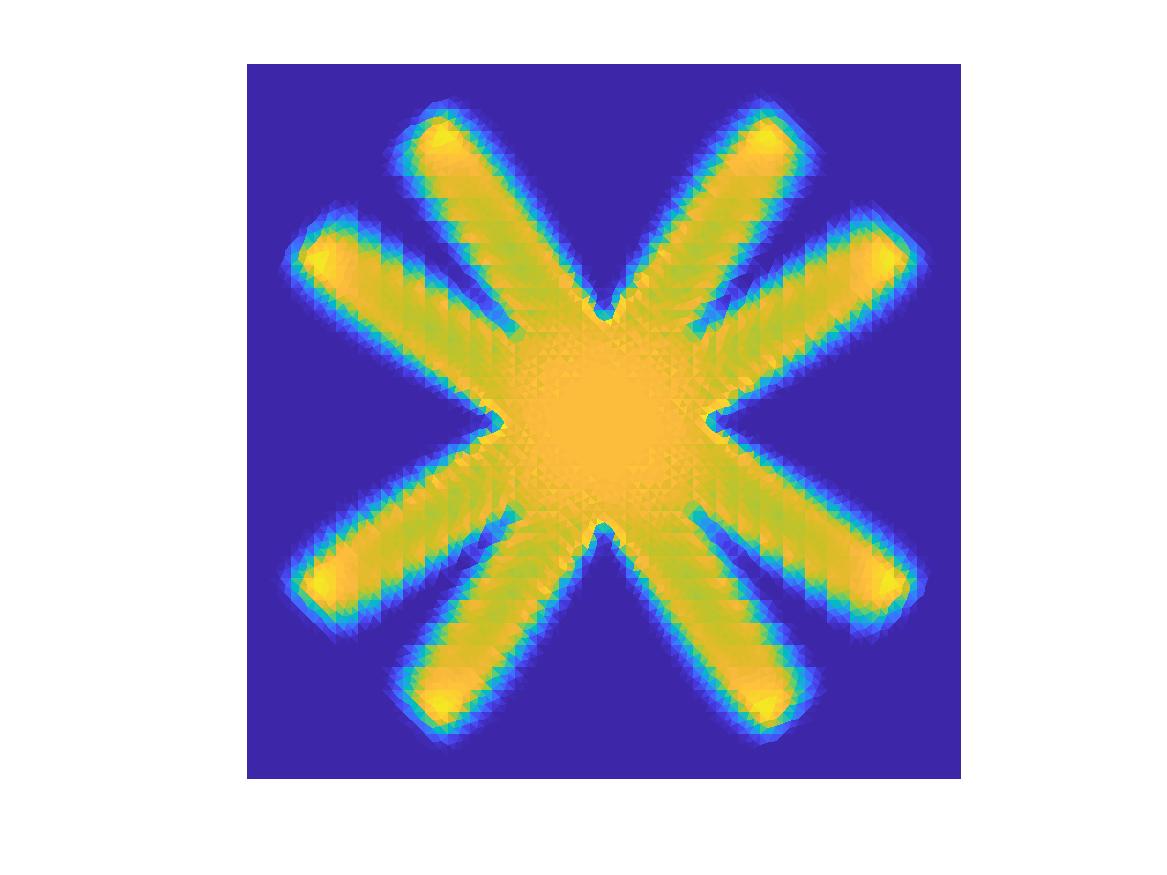}
			\includegraphics[trim={8cm 2.5cm 6.5cm 1.5cm},clip,width=0.30\textwidth]{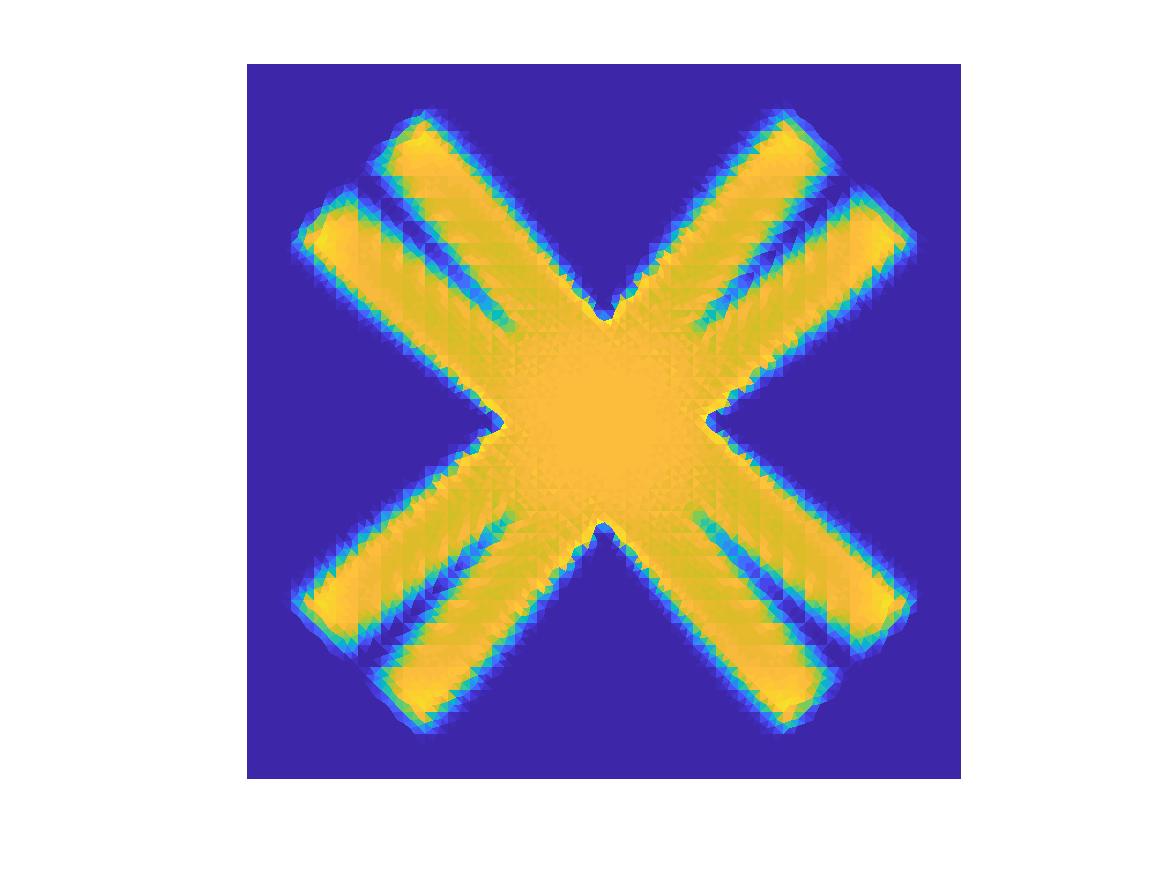}
			\includegraphics[trim={8cm 2.5cm 6.5cm 1.5cm},clip,width=0.30\textwidth]{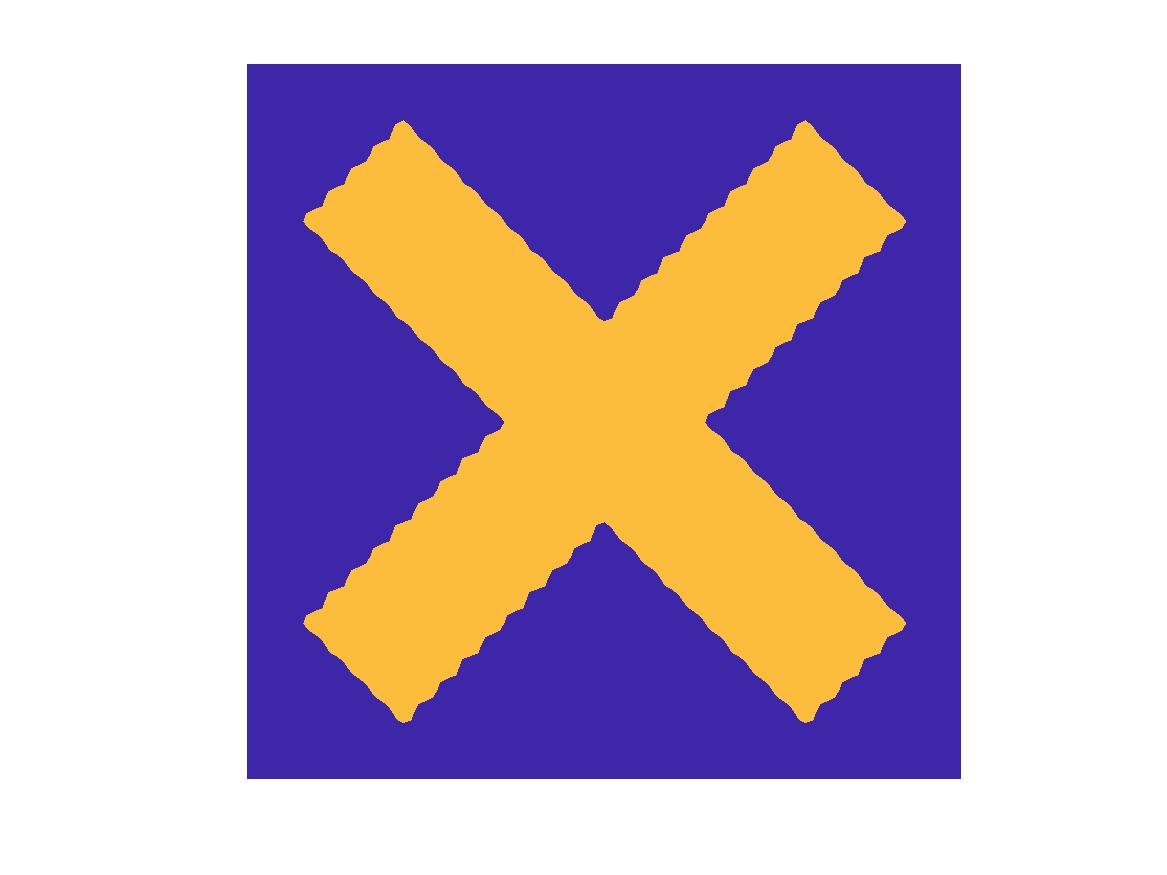}
		\end{minipage}
		\hspace{-2.5em}
		\begin{minipage}{0.09\textwidth}
		\centering
		\vspace{0.4em}
		\includegraphics[trim={16.5cm 0cm 1.4cm 0cm},clip,width=0.8cm,height=6cm]{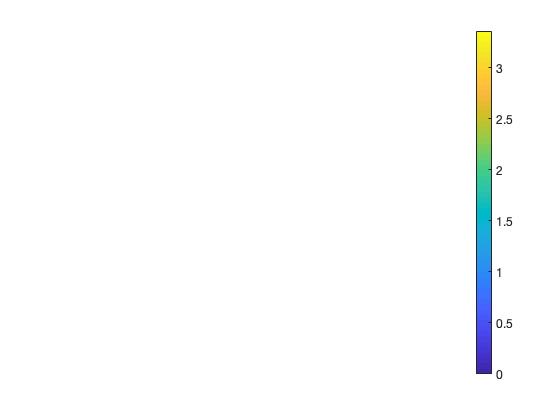}
		\end{minipage}
		\caption{Wasserstein interpolation between a cross distributed density and its rotation by $45$ degrees. Time increases from left to right, from top to bottom.}
		\label{fig:crossinterp}
	\end{figure}
	
	
	\section{Perspectives}\label{sec:perspectives}
	
	In this article we considered TPFA discretizations of the dynamical formulation of the quadratic optimal transport problem. In particular, we proposed a method based on nested meshes to deal with numerical instabilities that occur when using this type of techniques. We also proved quantitative convergence estimates in the case of smooth solutions and proposed the use of interior point techniques for the efficient numerical solutions of the scheme. Several interesting questions remain open on all the three aspects of the problem we considered:
	\begin{enumerate}
	  \item As for the issue of the numerical instabilities, the origin of these remains unclear, although their appearance is not surprising since the optimal transport interpolation does not imply any direct regularizing effect (e.g., the interpolation between two Dirac masses stays a Dirac). Our approach (together with previous works on the $L^1$ optimal transport problem \cite{facca1,facca2}) points towards the existence of a hidden inf-sup type of condition, analogous to the well-known ones for linear saddle point problems, which guaranties some regularity in the interpolation. 
	 \item The convergence results we proposed are only partial as they require that the density is strictly positive and also they do not apply to the density itself. Note, however, that the positivity requirement is only needed for the approximation result on the continuity equation in proposition \ref{prop:elliptic}, and this could be avoided using for example the regularization technique used by Lavenant in \cite{lavenant2019unconditional}. Note also that the same type of inf-sup condition needed for stability could also be used to derive convergence rates for the density.
	 \item The interior point technique we proposed for the solutions of the discrete system of optimality conditions can be made even more effective by using iterative methods for the solution of the linearized system. However, this is possible only once appropriate preconditioners are available. The challenging nature of the problem, which is mostly due to the interplay of the time and space discretization, implies that the design of such preconditioners requires a dedicated study and must be adapted to the discrete problem itself.  
	\end{enumerate}

	
	\section*{Acknowledgements}
	The authors are grateful to Mario Putti, Thomas Gallou\"et and Clément Cancès for interesting discussions and their suggestions on the subject.
	GT acknowledges that this project has received funding
	from the European Union’s Horizon 2020 research and innovation
	programme under the Marie Skłodowska-Curie grant agreement No 754362. AN acknowledges that this work was supported by a public grant as part of the Investissement
 d'avenir project, reference ANR-11-LABX-0056-LMH, LabEx LMH.
	\begin{center}
		\vspace{0.5em}
		\includegraphics[width=0.15\textwidth]{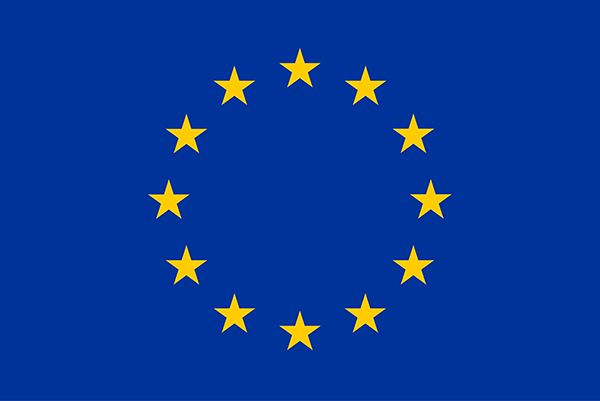}
	\end{center}
	
	\bibliographystyle{plain}      
	\bibliography{refs}

\end{document}